\documentclass[1p, preprint, 12pt]{elsarticle}

\usepackage{graphicx}
\usepackage{booktabs}
\usepackage{amsfonts}
\usepackage{amsmath,amssymb,amsthm,mathrsfs,amsopn}
\usepackage{a4wide}
\allowdisplaybreaks
\numberwithin{equation}{section}
\usepackage[pagewise]{lineno}\linenumbers\nolinenumbers
\usepackage{epstopdf}
\usepackage{float}
\usepackage{stfloats}
\usepackage{url}
\usepackage{color}
\usepackage[colorlinks, linkcolor=blue, citecolor=blue]{hyperref}
\usepackage{enumitem}

\newtheorem{theorem}{Theorem}[section]
\newtheorem{proposition}[theorem]{Proposition}
\newtheorem{corollary}[theorem]{Corollary}
\newtheorem{lemma}[theorem]{Lemma}
\theoremstyle{definition}

\newtheorem{remark}[theorem]{Remark}

\newcommand{\R}{\mathbb{R}}

\newcommand{\Rd}{\R^{d}}
\newcommand{\Rdp}{\R^{d+1}_{+}}
\newcommand{\Am}{A_{m}}
\newcommand{\eps}{\varepsilon}
\newcommand{\la}{\lambda}

\newcommand{\HH}{\mathcal{H}}
\newcommand{\VV}{\mathcal{V}}
\newcommand{\GG}{\mathscr{G}}
\newcommand{\Hs}{H^{1/2}}
\newcommand{\Hsr}{H^{1/2}_{\mathrm{rad}}}
\newcommand{\Hrad}{H^{1/2}_{\mathrm{rad}}}
\newcommand{\Ltwo}{L^{2}}
\newcommand{\Ltwor}{L^{2}_{\mathrm{rad}}}
\newcommand{\norm}[1]{\left\|#1\right\|}
\newcommand{\abs}[1]{\left|#1\right|}
\newcommand{\ip}[2]{\left\langle #1,#2\right\rangle}
\newcommand{\dx}{\,dx}
\newcommand{\dy}{\,dy}

\newcommand{\dxi}{\,d\xi}

\newcommand{\Ker}{\operatorname{ker}}
\newcommand{\sess}{\sigma_{\mathrm{ess}}}
\newcommand{\rad}{\mathrm{rad}}
\newcommand{\Uext}{\mathcal{U}}
\newcommand{\Gm}{G_{m}}
\newcommand{\Lp}{L_{+}}
\newcommand{\aform}{\mathfrak{a}}
\newcommand{\weakto}{\rightharpoonup}
\DeclareMathOperator{\Div}{div}

\newcommand{\ds}{\displaystyle}

\newcommand{\Hradd}{(H^{1/2}_{\mathrm{rad}})^{*}}

\DeclareMathOperator{\supp}{supp}
\begin{document}

\begin{frontmatter}

\title{Uniqueness of positive radial ground states for the massive pseudo-relativistic nonlinear Schr\"odinger equation}

\author[ustb]{Zongyan Lv}

\address[ustb]{School of Mathematics and Physics, University of Science and Technology  Beijing,  Beijing, 100083, P.R.China\\
	E-mail: \href{mailto:zongyanlv0535@163.com}{zongyanlv0535@163.com}}


\begin{abstract}
We study the semilinear pseudo-relativistic equation
\[
\sqrt{-\Delta+m^{2}}\,Q+\omega Q=Q^{p-1}\qquad\text{in }\R^{d},
\]
with $m>0$, $\omega>0$, $d\ge1$ and $2<p<2^{*}:=\frac{2d}{d-1}$ (with $2^{*}=\infty$ when $d=1$). We prove that the positive radial variational ground state is unique. The proof rests on a mass-covariant oscillation theorem for the linearized operator, established through a purely variational reduction: the Caffarelli--Silvestre--Duffin extension for $\sqrt{-\Delta+m^{2}}$ combined with a ground-state substitution reduces the linearized eigenvalue problem to a massless, potential-free weighted Steklov problem, whose second eigenfunction changes sign exactly once by strict superadditivity of the weighted Dirichlet energy under sign decomposition. The nondegeneracy of the linearization is then obtained from two orthogonality relations; the failure of scale invariance for $m>0$, which destroys the second relation used in the massless Frank--Lenzmann--Silvestre theory, is repaired by a mass-covariance identity. Global uniqueness follows by a continuation argument in the mass parameter, anchored at $m=0$ by the theorem of Frank--Lenzmann--Silvestre.
\end{abstract}


\end{frontmatter}

\section{Introduction}\label{sec:intro}

\subsection{The problem}

For $m>0$ let $\Am:=\sqrt{-\Delta+m^{2}}$ denote the operator with symbol
$\sqrt{\abs{\xi}^{2}+m^{2}}$, defined on $H^{1}(\Rd)$ and extended by duality to
$H^{-1/2}$. This is the free relativistic energy operator of a particle of mass $m$;
it arises as the Hamiltonian of the semirelativistic (pseudo-relativistic)
Schr\"odinger and Klein--Gordon equations. We are concerned with standing-wave profiles,
i.e.\ solutions of
\begin{equation}\label{eq:main}
\Am Q+\omega Q=Q^{p-1}\qquad\text{in }\Rd,\qquad Q>0,\quad Q\in\Hs(\Rd),
\end{equation}
with $\omega>0$ and
\begin{equation}\label{eq:range}
d\ge 1,\qquad 2<p<2^{*}:=\frac{2d}{d-1}\quad(2^{*}=\infty\text{ if }d=1).
\end{equation}
The exponent range \eqref{eq:range} is the $\Hs$-subcritical one: it is exactly the
range in which \eqref{eq:main} admits nontrivial $\Hs$ solutions and the associated
Weinstein (Gagliardo--Nirenberg) functional attains its infimum
\cite{CotiZelatiNolasco,ChoiSeok}.

Existence of a positive, radial, radially nonincreasing ground state for
\eqref{eq:main} is classical; see Coti~Zelati--Nolasco \cite{CotiZelatiNolasco} and
Choi--Seok \cite{ChoiSeok}. The much more delicate questions of \emph{uniqueness} and
\emph{nondegeneracy} of the ground state have, to the best of our knowledge, remained
open for fixed $m>0$ with a local power nonlinearity. The only prior results are
perturbative: Choi--Seok \cite{ChoiSeok} and Choi--Hong--Seok \cite{ChoiHongSeok}
established uniqueness and nondegeneracy in the nonrelativistic limit (the regime
$m\to\infty$ after rescaling), where \eqref{eq:main} converges to the classical NLS
$-\tfrac1{2m}\Delta u+\omega u=u^{p-1}$ for which uniqueness is due to Kwong
\cite{Kwong}. For the pseudo-relativistic Hartree nonlinearity, Lenzmann
\cite{LenzmannHartree} proved uniqueness under a smallness condition on the mass; the
local power case treated here is structurally different, since the nonlocal
nonlinearity there provides a compactness and monotonicity structure absent for the
local power.

Our main result closes this gap for all $m>0$ in the subcritical range.

\begin{theorem}[Existence of a positive radial ground state]
	\label{thm:existence}
	Let $d\ge1$, $m>0$, $\omega>0$ and $2<p<2^{*}=\frac{2d}{d-1}$
	(with $2^{*}=\infty$ if $d=1$). Then \eqref{eq:main} admits a ground state
	$Q\in H^{1/2}_{\mathrm{rad}}(\R^{d})$ that is positive, radial and radially
	nonincreasing. Equivalently, after the scaling normalization $\omega=1$, the sharp
	constant in the pseudo-relativistic Weinstein--Gagliardo--Nirenberg inequality
	\begin{equation}\label{eq:wgn}
		\|u\|_{p}\;\le\;C_{m,p}\,
		\big\langle\sqrt{-\Delta+m^{2}}\,u,u\big\rangle^{\theta/2}\,
		\|u\|_{2}^{\,1-\theta},
		\qquad
		\theta:=d\Big(\tfrac12-\tfrac1p\Big)\in(0,1),
	\end{equation}
	is attained, and every optimizer equals $Q$ up to a positive multiple and a dilation.
\end{theorem}

\begin{theorem}[Uniqueness]\label{thm:main}
Let $d\ge1$, $m>0$, $\omega>0$ and $2<p<2^{*}$. Then \eqref{eq:main} has a unique
positive radial variational ground state.
\end{theorem}

By scaling (Section~\ref{sec:scaling}) it suffices to treat the normalized equation
$\omega=1$ with a one-parameter family of masses; the parameter $m$ then plays the role
of a homotopy variable ranging over $[0,m_{*}]$, with $m=0$ corresponding to the
half-Laplacian equation
\begin{equation}\label{eq:massless}
(-\Delta)^{1/2}Q+Q=Q^{p-1},
\end{equation}
whose ground state is unique and nondegenerate by Frank--Lenzmann
(\cite{FrankLenzmann}, $d=1$) and Frank--Lenzmann--Silvestre
(\cite{FLS}, $d\ge1$).

\subsection{Strategy and the main difficulty}

\subsection{Organization}
Section~\ref{sec:prelim} fixes notation, records the extension theory and the sharp
decay of ground states, and reduces \eqref{eq:main} to the normalized family.
Section~\ref{sec:ratio} performs the ratio reduction $\Xi=\Phi_2/\Phi_1$.
Section~\ref{sec:steklov} develops the weighted Steklov framework and the spectral
correspondence $\nu_{k}=\la_{k+1}-\la_{1}$. Section~\ref{sec:osc} proves the
variational oscillation theorem. Section~\ref{sec:nondeg} proves radial nondegeneracy
via mass covariance. Section~\ref{sec:cift} develops the constrained implicit function
theorem. Section~\ref{sec:local} gives local uniqueness, and
Sections~\ref{sec:global}--\ref{sec:seam} the global continuation.
\section{Preliminaries}\label{sec:prelim}

\subsection{Notation and function spaces}

 We define  $\ds\norm u_{\Hs}^2=\int_{\Rd}(1+\abs\xi^2)^{1/2}\abs{\hat u}^2$.
The radial subspace is $\Hsr(\Rd)$; the radial $L^{2}$ space is $\Ltwor(\Rd)$. Throughout,
$C$ denotes a constant that may change from line to line and depends only on $d,p$ and
where indicated on $m_{*}$. We use $\varrho:=\sqrt{\abs{x}^{2}+y^{2}}$ for points
$(x,y)\in\Rdp=\{(x,y):x\in\Rd,\ y>0\}$, and $r=\abs{x}$. We denote by
$\mathcal{K}_{\nu}$ the modified Bessel function of the second kind.

The operator $\Am$ is self-adjoint on $\Ltwo(\Rd)$ with form domain $\Hs(\Rd)$ and
\begin{equation}\label{eq:Am-form}
\ip{\Am u}{u}=\int_{\Rd}\sqrt{\abs{\xi}^{2}+m^{2}}\,\abs{\hat u(\xi)}^{2}\dxi
=\tfrac12\iint_{\Rd\times\Rd}\abs{u(x)-u(y)}^{2}J_{m}(\abs{x-y})\dx\dy
+m\norm{u}_{2}^{2},
\end{equation}
where the Yukawa-type kernel is
\begin{equation}\label{eq:Jm}
J_{m}(r)=c_{d}\,m^{\frac{d+1}{2}}\,r^{-\frac{d+1}{2}}\,\mathcal{K}_{\frac{d+1}{2}}(mr)>0,
\end{equation}
for a constant $c_d>0$. The representation \eqref{eq:Jm} is classical (Fourier inversion
of $\sqrt{\abs{\xi}^{2}+m^{2}}-m$; see \cite[\S7.11]{LiebLoss}).

We record for later use that $J_{m}>0$ everywhere and $J_{m}$ is strictly decreasing.
Moreover, inserting the asymptotics $\mathcal{K}_{\nu}(z)\sim\Gamma(\nu)2^{\nu-1}z^{-\nu}$
as $z\to0$ and $\mathcal{K}_{\nu}(z)\sim\sqrt{\pi/(2z)}\,e^{-z}$ as $z\to\infty$ into
\eqref{eq:Jm} gives
\begin{equation}\label{eq:Jm-asy}
	J_{m}(r)\sim c_{d}\,r^{-(d+1)}\quad \text{as} ~r\to0,
	\qquad
	J_{m}(r)\sim C\,r^{-\frac{d+2}{2}}e^{-mr}\quad \text{as} ~r\to\infty.
\end{equation}

\subsection{Sharp estimate of ground states via the Green kernel}\label{sec:decay}

\begin{proposition}\label{prop:Jm-gauss}
	Let $m>0$, 
	\[
	\nu_m(\tau):=\frac{m}{\sqrt{4\pi}}\,\tau^{-3/2}e^{-m^2\tau}
	\qquad
	g_\tau(r):=(4\pi\tau)^{-d/2}e^{-r^2/(4\tau)},
	\qquad
	d\mu_m:=\nu_m\,d\tau\ge0 .
	\]
	Then
	\begin{equation}\label{eq:Jm-gauss}
		J_m(\abs z)=\int_0^\infty g_\tau(\abs z)\,d\mu_m(\tau),
	\end{equation}
	and for each $\tau>0$, $g_\tau$ is radial, strictly decreasing, $\int_{\Rd}g_\tau=1$.
\end{proposition}

\begin{proof}
	Throughout we use the convention
	\begin{equation}\label{eq:convention}
		\hat f(\xi)=\int_{\Rd}f(z)e^{-iz\cdot\xi}\dx z,
		\qquad
		\mathcal F^{-1}[f](z)=(2\pi)^{-d}\int_{\Rd}f(\xi)e^{iz\cdot\xi}\dxi ,
	\end{equation}
	for which Plancherel's theorem reads
	$\ip{f}{g}=(2\pi)^{-d}\ip{\hat f}{\hat g}$, and \eqref{eq:Am-form} takes the form
	$\ip{\Am u}{u}=(2\pi)^{-d}\int_{\Rd}\sqrt{\abs\xi^2+m^2}\,\abs{\hat u(\xi)}^2\dxi$.
	
	\emph{Step 1.} For $\psi_m(\lambda):=\sqrt{\lambda+m^2}-m$, we have
	\[
	\psi_m(0)=0,\qquad
	\psi_m'(\lambda)=\tfrac12(\lambda+m^2)^{-1/2}>0,\qquad
	\psi_m''(\lambda)=-\tfrac14(\lambda+m^2)^{-3/2}<0,
	\]
	\[
	(-1)^{n-1}\psi_m^{(n)}\ge0\ (n\ge1),
	\quad \text{and} \quad
	\lim_{\lambda\to\infty}\frac{\psi_m(\lambda)}{\lambda}=0 .
	\]
	Hence $\psi_m$ is a complete Bernstein function with vanishing killing and drift terms. From \cite[Ch.~5]{SSV}, we have
	\begin{equation}\label{eq:bernstein}
		\sqrt{\lambda+m^2}-m=\int_0^\infty\big(1-e^{-\tau\lambda}\big)\,\nu_m(\tau)\,d\tau,
		\qquad\lambda\ge0 .
	\end{equation}
	Setting $\lambda=\abs\xi^2$, we see that
	\begin{equation}\label{eq:symbol-decomp}
		\sqrt{\abs\xi^2+m^2}-m=\int_0^\infty\big(1-e^{-\tau\abs\xi^2}\big)\,\nu_m(\tau)\,d\tau .
	\end{equation}
	
	\emph{Step 2.} It follows from \eqref{eq:convention} that
	\[
	\mathcal F^{-1}\big[e^{-\tau\abs\xi^2}\big](z)
	=(2\pi)^{-d}\prod_{j=1}^{d}\int_{\R}e^{-\tau\xi_j^2+iz_j\xi_j}\,d\xi_j
	=(2\pi)^{-d}\Big(\tfrac{\pi}{\tau}\Big)^{d/2}e^{-\abs z^2/(4\tau)}
	=g_\tau(\abs z),
	\]
	where each factor is computed by completing the square,
	$$\int_{\R}e^{-\tau\xi_j^2+iz_j\xi_j}d\xi_j
	=e^{-z_j^2/(4\tau)}\int_{\R}e^{-\tau(\xi_j-iz_j/(2\tau))^2}d\xi_j
	=\sqrt{\pi/\tau}\,e^{-z_j^2/(4\tau)}.$$
	Applying $\mathcal F$ to both sides and evaluating at $\xi=0$, we have
	\begin{equation}\label{eq:gauss-fourier}
		\widehat{g_\tau}(\xi)=e^{-\tau\abs\xi^2}
		\quad \text{and} \quad
		\widehat{g_\tau}(0)=\int_{\Rd}g_\tau=1 ,
	\end{equation}
	the last identity because $\hat k(0)=\int_{\Rd}k(z)e^{-iz\cdot0}\dx z=\int_{\Rd}k$ for any
	$k\in L^1(\Rd)$.
	
	\emph{Step 3.} Let $u\in C_c^\infty(\Rd)$ be real and let $k\in L^1(\Rd)$ be symmetric.
	Expanding $\abs{u(x)-u(y)}^2=\abs{u(x)}^2+\abs{u(y)}^2-2u(x)u(y)$ and using the
	translation invariance, we obtain
	\begin{equation}\label{eq:khat0}
		\int_{\Rd}k(x-y)\dy=\int_{\Rd}k(z)\dx z=\hat k(0),
	\end{equation}
	Fubini's thorem gives, for the first two terms,
	\[
	\tfrac12\iint\big(\abs{u(x)}^2+\abs{u(y)}^2\big)k(x-y)\dx\dy=\hat k(0)\,\norm u_2^2 ,
	\]
	while the cross term is a convolution, so that by $\widehat{k*u}=\hat k\,\hat u$ and
	Plancherel's theorem,
	\[
	-\iint u(x)k(x-y)u(y)\dx\dy=-\ip{k*u}{u}
	=-(2\pi)^{-d}\int_{\Rd}\hat k(\xi)\abs{\hat u(\xi)}^2\dxi .
	\]
	Since $\norm u_2^2=(2\pi)^{-d}\int_{\Rd}\abs{\hat u}^2\dxi$, adding the two displays yields
	\begin{equation}\label{eq:diff-symbol}
		\tfrac12\iint_{\Rd\times\Rd}\abs{u(x)-u(y)}^2 k(x-y)\dx\dy
		=(2\pi)^{-d}\int_{\Rd}\big(\hat k(0)-\hat k(\xi)\big)\abs{\hat u(\xi)}^2\dxi .
	\end{equation}
	By \eqref{eq:gauss-fourier}, \eqref{eq:diff-symbol} with $k=g_\tau$ gives
	\begin{equation}\label{eq:gauss-form}
		\tfrac12\iint\abs{u(x)-u(y)}^2 g_\tau(\abs{x-y})\dx\dy
		=(2\pi)^{-d}\int_{\Rd}\big(1-e^{-\tau\abs\xi^2}\big)\abs{\hat u(\xi)}^2\dxi .
	\end{equation}
	Integrating \eqref{eq:gauss-form} in $\tau$ against $\nu_m\ge0$, using \eqref{eq:symbol-decomp},  Tonelli's theorem,
	Plancherel's theorem and \eqref{eq:Am-form}, then, we obtain
	\begin{align}
		\tfrac12\iint\abs{u(x)-u(y)}^2\Big(\int_0^\infty g_\tau(\abs{x-y})\,\nu_m(\tau)\,d\tau\Big)\dx\dy
		&=(2\pi)^{-d}\!\int_0^\infty\!\!\int_{\Rd}\big(1-e^{-\tau\abs\xi^2}\big)\abs{\hat u}^2\nu_m(\tau)\dxi\,d\tau
		\notag\\
		&=(2\pi)^{-d}\!\int_{\Rd}\big(\sqrt{\abs\xi^2+m^2}-m\big)\abs{\hat u(\xi)}^2\dxi
		\notag\\
		&=\ip{(\Am-m)u}{u} \notag\\
		&=\tfrac12\iint\abs{u(x)-u(y)}^2 J_m(\abs{x-y})\dx\dy .
		\label{eq:two-forms}
	\end{align}
	As \eqref{eq:two-forms} holds for all $u\in C_c^\infty(\Rd)$ and both kernels are radial,
	they agree a.e., which is \eqref{eq:Jm-gauss}.
	
	\emph{Step 4.} $g_\tau(z)=(4\pi\tau)^{-d/2}e^{-\abs z^2/(4\tau)}$ depends on $z$ only
	through $\abs z$, is strictly decreasing in $\abs z$, and $\int_{\Rd}g_\tau=1$ by
	\eqref{eq:gauss-fourier}.
\end{proof}

\begin{proposition}[Green kernel of $(\Am+1)^{-1}$]\label{prop:green}
	Let $\Gm$ denote the kernel of $(\Am+1)^{-1}$ on $\Rd$, i.e.\ the radial positive solution
	of $(\Am+1)\Gm=\delta_{0}$. Then $\Gm>0$, $\Gm$ is radially strictly decreasing,
	$\Gm\in L^{1}(\Rd)$, and there are constants $C_{d,m}\ge1$, $R_0\ge1$ such that
	\begin{equation}\label{eq:green}
		C_{d,m}^{-1}\,\abs x^{-\frac{d+1}{2}}e^{-m\abs x}
		\;\le\;\Gm(x)\;\le\;
		C_{d,m}\,\abs x^{-\frac{d+1}{2}}e^{-m\abs x},
		\qquad \abs x\ge R_0 .
	\end{equation}
	Near the origin, $\Gm(x)\simeq\abs x^{-(d-1)}$ for $d\ge2$ and
	$\Gm(x)\simeq\log\tfrac1{\abs x}$ for $d=1$. Moreover, for every $m'<m$, there holds that 
	\begin{equation}\label{eq:green-upper}
		\Gm(x)\le C\,e^{-m'\abs x},\qquad \abs x\ge1.
	\end{equation}
\end{proposition}

\begin{proof}
	\emph{Subordination: positivity, strict monotonicity, integrability.}
	By Proposition~\ref{prop:Jm-gauss}, $\psi_m(\lambda)=\sqrt{\lambda+m^2}-m$ is a complete
	Bernstein function, so $\Am-m$ is subordinate to $-\Delta$ and the semigroup $e^{-t\Am}$
	has kernel
	\begin{equation}\label{eq:subord}
		p_{t}^{m}(x)=\int_{0}^{\infty}\frac{e^{-\abs{x}^{2}/(4\tau)}}{(4\pi\tau)^{d/2}}\,
		\eta_{t}^{m}(\tau)\,d\tau ,
		\qquad \eta_{t}^{m}>0\ \text{ on }(0,\infty),
	\end{equation}
	where $\eta_t^m$ is the transition density of the relativistic $\tfrac12$-stable
	subordinator \cite[Ch.~5, Ch.~13]{SSV}. Each Gaussian in \eqref{eq:subord} is positive and
	strictly decreasing in $\abs x$, and $\eta_t^m>0$. Thus, $p_t^m>0$ and $p_t^m$ is strictly
	decreasing in $\abs x$. Writing the resolvent as the Laplace transform of the semigroup,
	\begin{equation}\label{eq:resolvent-laplace}
		\Gm(x)=\int_{0}^{\infty}e^{-t}\,p_{t}^{m}(x)\,dt ,
	\end{equation}
	these properties are inherited by $\Gm$: $\Gm>0$ and $\Gm$ is radially strictly
	decreasing. Since $\Gm\ge0$,
	\[
	\int_{\Rd}\Gm=\widehat{\Gm}(0)=\big(\sqrt{0+m^{2}}+1\big)^{-1}=(m+1)^{-1}<\infty ,
	\]
	then $\Gm\in L^{1}(\Rd)$.
	
	\emph{Decay.} The two-sided bound \eqref{eq:green} is the exponential-decay regime for the
	resolvent kernel $\int_0^\infty e^{-t}p_t^m(x)\,dt$ of a relativistic operator,
	established by Carmona--Masters--Simon \cite{CMS}; see also Kaleta--Schilling--Sztonyk
	\cite{KSS} for sharp two-sided estimates of resolvent kernels of L\'evy operators covering
	this case. The bound \eqref{eq:green-upper} follows from the upper bound in
	\eqref{eq:green} by absorbing the algebraic prefactor into $e^{-(m-m')\abs x}$. The
	behaviour near the origin follows from \eqref{eq:Jm-asy}: for small $\abs x$ both the mass
	and the ``$+1$'' are lower-order perturbations, so $\Gm$ inherits the leading singularity
	of the kernel of $(-\Delta)^{-1/2}$.
\end{proof}

\begin{remark}\label{rem:green}
	Only the exponential rate $m$ and the algebraic prefactor $\abs x^{-\frac{d+1}{2}}$ in
	\eqref{eq:green} are used below; no exact constant is needed. Accordingly all decay
	statements in this paper are formulated as two-sided bounds rather than exact asymptotics.
\end{remark}

\begin{lemma}\label{lem:convolution}
	Let $H,F\ge0$ be radial, $H,F\in L^1(\Rd)$ and $F\not\equiv0$. Assume  
	\begin{equation}\label{eq:conv-hyp}
		c_H^{-1}\abs x^{-\frac{d+1}{2}}e^{-m\abs x}\le H(x)\le c_H\abs x^{-\frac{d+1}{2}}e^{-m\abs x},
		\ \ \abs x\ge R_0
		\quad \text{and} \quad
		F(x)\le c_F\,e^{-m''\abs x},\ \ x\in\Rd,
	\end{equation}
	for some	$c_H,c_F\ge1$, $R_0\ge1$, $m''>m$.
	Then there are $C\ge1$, $R_1\ge R_0$ with
	\begin{equation}\label{eq:conv-concl}
		C^{-1}\abs x^{-\frac{d+1}{2}}e^{-m\abs x}\le(H*F)(x)\le C\abs x^{-\frac{d+1}{2}}e^{-m\abs x},
		\qquad \abs x\ge R_1 .
	\end{equation}
\end{lemma}

\begin{proof}
	We write $(H*F)(x)=I_1(x)+I_2(x)$, where
	\[
	I_1(x):=\int_{\{\abs{x-y}>\abs x/2\}}H(x-y)F(y)\dy,
	\qquad
	I_2(x):=\int_{\{\abs{x-y}\le\abs x/2\}}H(x-y)F(y)\dy,
	\]
	and let $\abs x\ge2R_0$.
	
On $\{\abs{x-y}>\abs x/2\}$ we have $\abs{x-y}>\abs x/2\ge R_0$,
	so \eqref{eq:conv-hyp} applies and, together with $\abs{x-y}\ge\abs x-\abs y$,
	\[
	H(x-y)\le c_H\abs{x-y}^{-\frac{d+1}{2}}e^{-m\abs{x-y}}
	\le c_H2^{\frac{d+1}{2}}\abs x^{-\frac{d+1}{2}}e^{-m\abs x}e^{m\abs y}.
	\]
	Since $F(y)e^{m\abs y}\le c_Fe^{-(m''-m)\abs y}\in L^1(\Rd)$,
	\[
	I_1(x)\le c_H2^{\frac{d+1}{2}}\abs x^{-\frac{d+1}{2}}e^{-m\abs x}\int_{\Rd}F(y)e^{m\abs y}\dy .
	\]
	
 Fix $\rho\ge1$ with $\int_{\abs y\le\rho}F>0$. For
	$\abs x\ge\max(4R_0,4\rho)$ and $\abs y\le\rho$ we have
	$\abs x/2\le\abs x-\rho\le\abs{x-y}\le\abs x+\rho\le2\abs x$, so $\{\abs y\le\rho\}$ is
	contained in the domain of $I_1$ and
	\[
	H(x-y)\ge c_H^{-1}\abs{x-y}^{-\frac{d+1}{2}}e^{-m\abs{x-y}}
	\ge c_H^{-1}2^{-\frac{d+1}{2}}e^{-m\rho}\abs x^{-\frac{d+1}{2}}e^{-m\abs x},
	\]
	whence
	$I_1(x)\ge c_H^{-1}2^{-\frac{d+1}{2}}e^{-m\rho}\big(\int_{\abs y\le\rho}F\big)
	\abs x^{-\frac{d+1}{2}}e^{-m\abs x}$.
	
 On $\{\abs{x-y}\le\abs x/2\}$ we have
	$\abs y\ge\abs x-\abs{x-y}\ge\abs x/2$, so $F(y)\le c_Fe^{-m''\abs y}\le c_Fe^{-m''\abs x/2}$,
	and since $H\in L^1(\Rd)$,
	\[
	I_2(x)\le c_F\,e^{-m''\abs x/2}\int_{\Rd}H
	=c_F\norm H_{L^1}\,e^{-m''\abs x/2}.
	\]
Using $e^{-m''\abs y}=e^{-m\abs y}e^{-(m''-m)\abs y}$ and 
	$\abs y\ge\abs x/2$ together with $\abs{x-y}\le\abs x/2$, we get
	$\abs{x-y}+\abs y\ge\abs x$, hence
	\[
	I_2(x)\le c_F\,e^{-m\abs x}e^{-\frac{m''-m}{2}\abs x}\norm H_{L^1}
	=o\big(\abs x^{-\frac{d+1}{2}}e^{-m\abs x}\big)
	\]
	as $\abs x\to\infty$, since $m''>m$. Combining the three estimates yields
	\eqref{eq:conv-concl}.
\end{proof}

\subsection{The variational problem}\label{sec:scaling}

If $Q$ solves \eqref{eq:main} with parameters $(m,\omega)$, set
$P(x):=\omega^{-1/(p-2)}Q(\omega^{-1}x)$.
 Using
$$\Am[Q(\omega^{-1}\cdot)](x)=\omega^{-1}(A_{m/\omega}Q)(\omega^{-1}x)$$ one checks that
$P$ solves
\begin{equation}\label{eq:norm}
	\sqrt{-\Delta+\widetilde m^{2}}\,P+P=P^{p-1},\qquad \widetilde m:=\frac{m}{\omega}>0.
\end{equation}
Conversely every solution of \eqref{eq:norm} yields one of \eqref{eq:main}. Thus
Theorem~\ref{thm:main} is equivalent to: for every $\widetilde m>0$,
\eqref{eq:norm} has a unique positive radial variational ground state. We henceforth
study \eqref{eq:norm}, drop the tilde, and regard $m\in[0,m_{*}]$ as a homotopy
parameter for an arbitrary but fixed $m_{*}>0$; the case $m=0$ is \eqref{eq:massless}.

Define on $\Hs(\Rd)$,
\begin{equation}\label{eq:action}
S_{m}(u):=\tfrac12\ip{\Am u}{u}+\tfrac12\norm{u}_{2}^{2}-\tfrac1p\norm{u}_{p}^{p},
\end{equation}
with Nehari manifold
\begin{equation}\label{eq:nehari}
\mathcal{N}_{m}:=\Big\{u\in\Hs(\Rd)\setminus\{0\}:\
\ip{\Am u}{u}+\norm{u}_{2}^{2}=\norm{u}_{p}^{p}\Big\},
\qquad
m_{m}:=\inf_{\mathcal N_{m}}S_{m}.
\end{equation}
A variational ground state is a minimizer of $S_{m}$ on $\mathcal N_{m}$; equivalently,
after the fibering normalization below, a minimizer of the Weinstein--Gagliardo--Nirenberg
quotient. We denote the set of positive radial variational ground states of
\eqref{eq:norm} by $\mathcal{G}_{m}$.

\begin{proof}[\bf\text{Proof of Theorem \ref{thm:existence}}]
	By the scaling,  it suffices to produce a minimizer
	for $m$ fixed and $\omega=1$. The case of general $\omega$ follows by the transformation
	$P(x)=\omega^{-1/(p-2)}Q(\omega^{-1}x)$. We construct the ground state as a minimizer of a
	constrained variational problem, then upgrade it to a positive radial nonincreasing
	solution of \eqref{eq:norm}.
	
	\medskip
\emph{Step 1.}
Define
\begin{equation}\label{eq:Sinf}
	S:=\inf\Big\{\,B(u):\ u\in\Hsr(\Rd),\ \norm u_p=1\,\Big\},
	\qquad B(u):=\ip{\Am u}{u}+\norm u_2^2 .
\end{equation}
By the  Fourier transform, with $\ds\norm u_{\Hs}^2=\int_{\Rd}(1+\abs\xi^2)^{1/2}\abs{\hat u}^2$, we have
\[
B(u)=\int_{\Rd}\big(\sqrt{\abs\xi^2+m^2}+1\big)\,\abs{\hat u(\xi)}^2\,d\xi .
\]
We claim the two weights $w(\xi):=\sqrt{\abs\xi^2+m^2}+1$ and $(1+\abs\xi^2)^{1/2}$ are
pointwise comparable, with constants depending only on $m$:
\begin{equation}\label{eq:weight}
	\tfrac{1}{2}\,(1+\abs\xi^2)^{1/2}\;\le\;w(\xi)\;\le\;(m+2)\,(1+\abs\xi^2)^{1/2},
	\qquad \xi\in\Rd .
\end{equation}

If $\abs\xi\ge1$, then
$w(\xi)\ge\sqrt{\abs\xi^2}=\abs\xi$.  Since $1\le\abs\xi^2$, we have
\[
(1+\abs\xi^2)^{1/2}\le(2\abs\xi^2)^{1/2}=\sqrt2\,\abs\xi\le2\abs\xi ,
\]
whence $w(\xi)\ge\abs\xi\ge\tfrac12(1+\abs\xi^2)^{1/2}$. 

If $\abs\xi\le1$, then
$w(\xi)\ge1$, while
$(1+\abs\xi^2)^{1/2}\le\sqrt2\le2$. Thus, we have 
$w(\xi)\ge1\ge\tfrac12(1+\abs\xi^2)^{1/2}$.
Using $\sqrt{a+b}\le\sqrt a+\sqrt b$, we have 
\[
w(\xi)=\sqrt{\abs\xi^2+m^2}+1\le\abs\xi+m+1.
\]
Combing $\abs\xi\le(1+\abs\xi^2)^{1/2}$ with   $1\le(1+\abs\xi^2)^{1/2}$, it follows that
\[
w(\xi)\le(1+\abs\xi^2)^{1/2}+(m+1)(1+\abs\xi^2)^{1/2}=(m+2)(1+\abs\xi^2)^{1/2} .
\]

Multiplying \eqref{eq:weight} by $\abs{\hat u(\xi)}^2\ge0$ and integrating yields the norm
equivalence
\begin{equation}\label{eq:B-equiv}
	\tfrac12\,\norm u_{\Hs}^2\;\le\;B(u)\;\le\;(m+2)\,\norm u_{\Hs}^2 ,
	\qquad u\in\Hs(\Rd).
\end{equation}
By the Sobolev embedding $\Hs(\Rd)\hookrightarrow L^p(\Rd)$ for $2\le p<2^*$ and \eqref{eq:B-equiv},
\[
\norm u_p\le C\norm u_{\Hs}\le C\sqrt2\,B(u)^{1/2}=:C'B(u)^{1/2}.
\]
Restricting to the constraint $\norm u_p=1$ gives $1\le C'B(u)^{1/2}$, i.e.\
$B(u)\ge(C')^{-2}$. Taking  the infimum, we obtain $S\ge(C')^{-2}>0$. Hence $S>0$ and any minimizing
sequence $(u_n)\subset\Hsr(\Rd)$ (i.e.\ $\norm{u_n}_p=1$, $B(u_n)\to S$) satisfies
$\norm{u_n}_{\Hs}^2\le2B(u_n)\le C''$ by \eqref{eq:B-equiv}, so $(u_n)$ is bounded in
$\Hsr(\Rd)$.
	
	\medskip
	\emph{Step 2.}
Since the radial embedding $	\Hsr(\Rd)\hookrightarrow\hookrightarrow L^p(\Rd)$
	is compact (Lions' radial compactness for $H^{1/2}$; see \cite[Prop.~3.5]{Secchi} for the
	Bessel formulation, or Sickel--Skrzypczak for the general fractional radial embedding).
	By Step~1, $(u_n)$ is bounded in $\Hsr$, passing to a subsequence, we have 
	\[
	u_n\weakto Q\ \text{ in }\Hsr(\Rd),\qquad u_n\to Q\ \text{ in }L^p(\Rd).
	\]
	Strong $L^p$-convergence gives $\norm Q_p=\lim_n\norm{u_n}_p=1$, so $Q\not\equiv0$ and $Q$
	is admissible for \eqref{eq:Sinf}. The functional $B$ is a nonnegative energy functional, then,
by weakly lower semicontinuous,  we  have 
	\[
	B(Q)\le\liminf_{n\to\infty}B(u_n)=S .
	\]
	Since $Q$ is admissible, $B(Q)\ge S$. Therefore $B(Q)=S$ and $Q$ attains the infimum.
	
	\medskip
\emph{Step 3.}
We show the minimizer may be taken nonnegative, radial and radially nonincreasing.

 Passing from $Q$ to $\abs Q$ preserves the constraint
$\norm{\abs Q}_p=1$ and $\norm{\abs Q}_2=\norm Q_2$. Then it does not increase the energy functional by the pointwise inequality $\big|\abs{Q(x)}-\abs{Q(y)}\big|\le\abs{Q(x)-Q(y)}$ and
the singular-integral representation \eqref{eq:Am-form},
\[
\ip{\Am\abs Q}{\abs Q}
=\tfrac12\iint\big|\abs{Q(x)}-\abs{Q(y)}\big|^2 J_m(\abs{x-y})\,dx\,dy+m\norm Q_2^2
\le\ip{\Am Q}{Q}.
\]
Thus, we have  $B(\abs Q)\le B(Q)=S$, and admissibility forces $B(\abs Q)=S$. We may thus assume
$Q\ge0$.

 Let $Q^*$ be the symmetric-decreasing rearrangement of
$Q\ge0$. It preserves every $L^q$-norm, so $\norm{Q^*}_p=1$ and $\norm{Q^*}_2=\norm Q_2$.
We claim the \emph{pseudo-relativistic P\'olya--Szeg\H{o} inequality}
\begin{equation}\label{eq:polya}
	\ip{\Am Q^*}{Q^*}\le\ip{\Am Q}{Q}.
\end{equation}

Firstly, by the singular-integral form \eqref{eq:Am-form}, we obtain
\[
\ip{\Am u}{u}=\tfrac12\,\mathcal D[u]+m\norm u_2^2,\qquad
\mathcal D[u]:=\iint_{\Rd\times\Rd}\abs{u(x)-u(y)}^2 J_m(\abs{x-y})\,dx\,dy .
\]
Since rearrangement preserves the mass term, $\norm{Q^*}_2^2=\norm Q_2^2$, inequality
\eqref{eq:polya} is equivalent to
\begin{equation}\label{eq:diff-PS}
	\mathcal D[Q^*]\le\mathcal D[Q].
\end{equation}

Secondly, fix $\tau>0$ and each  $g_\tau$ is a symmetric-decreasing $L^1$ kernel, so the
rearrangement inequality of Frank--Seiringer \cite[Lem.~A.2]{FrankSeiringer}, applied with
the convex function $J(t)=t^2$ and the kernel $k=g_\tau$, gives
\begin{equation}\label{eq:FS-tau}
	\mathcal D_\tau[Q^*]\le\mathcal D_\tau[Q],\qquad
	\mathcal D_\tau[u]:=\iint\abs{u(x)-u(y)}^2 g_\tau(\abs{x-y})\,dx\,dy .
\end{equation}

Thirdly, The integrand
$(x,y,\tau)\mapsto\abs{u(x)-u(y)}^2 g_\tau(\abs{x-y})$ is nonnegative and measurable, so by
Tonelli's theorem the order of integration may be exchanged. Using the Gaussian
superposition \eqref{eq:Jm-gauss}, we obtain
\begin{align}
	\mathcal D[u]
	&=\iint\abs{u(x)-u(y)}^2 J_m(\abs{x-y})\,dx\,dy \notag\\ 
	&=\iint\abs{u(x)-u(y)}^2\Big(\int_0^\infty g_\tau(\abs{x-y})\,d\mu_m(\tau)\Big)dx\,dy
	\label{eq:insert}\\    
	&=\int_0^\infty\Big(\iint\abs{u(x)-u(y)}^2 g_\tau(\abs{x-y})\,dx\,dy\Big)d\mu_m(\tau)
	=\int_0^\infty\mathcal D_\tau[u]\,d\mu_m(\tau).   	\label{eq:tonelli}
\end{align}
Applying \eqref{eq:tonelli} to both $Q$ and $Q^*$ and using \eqref{eq:FS-tau} together
with $d\mu_m\ge0$, we have 
\[
\mathcal D[Q^*]=\int_0^\infty\mathcal D_\tau[Q^*]\,d\mu_m(\tau)
\le\int_0^\infty\mathcal D_\tau[Q]\,d\mu_m(\tau)=\mathcal D[Q],
\]
which is \eqref{eq:diff-PS}, hence \eqref{eq:polya}.

Consequently, $B(Q^*)=\ip{\Am Q^*}{Q^*}+\norm{Q^*}_2^2\le\ip{\Am Q}{Q}+\norm Q_2^2=B(Q)=S$.
Since $\norm{Q^*}_p=1$ forces $B(Q^*)\ge S$, we get $B(Q^*)=S$. Thus $Q^*$ is a minimizer,
radial and radially nonincreasing, and we rename $Q:=Q^*$.

\medskip
\emph{Step 4.}
The minimizer $Q$ of \eqref{eq:Sinf} satisfies, with a Lagrange multiplier $\Lambda$, the
weak equation
\begin{equation}\label{eq:EL}
	\ip{\Am Q}{\varphi}+\ip{Q}{\varphi}=\Lambda\int_{\Rd}Q^{p-1}\varphi ,
	\qquad\forall\,\varphi\in\Hsr(\Rd).
\end{equation}
Testing \eqref{eq:EL} with $\varphi=Q$ and using $\norm Q_p^p=1$ fixes
\begin{equation}\label{eq:multiplier}
	\Lambda=\ip{\Am Q}{Q}+\norm Q_2^2=B(Q)=S>0 .
\end{equation}
Thus $Q$ solves $(\Am+1)Q=\Lambda\,Q^{p-1}$ weakly.

\medskip
\emph{Step 5.}
Set $\widetilde Q:=\Lambda^{1/(p-2)}Q$. Using \eqref{eq:EL}, we obtain 
\[
(\Am+1)\widetilde Q=\Lambda^{1/(p-2)}(\Am+1)Q=\Lambda^{1/(p-2)}\Lambda\,Q^{p-1}
=\Lambda^{1+\frac1{p-2}}Q^{p-1},
\]
while $\widetilde Q^{\,p-1}=\Lambda^{(p-1)/(p-2)}Q^{p-1}$. Since
$1+\tfrac1{p-2}=\tfrac{p-1}{p-2}$, the two exponents coincide and
\[
(\Am+1)\widetilde Q=\widetilde Q^{\,p-1},
\]
which is \eqref{eq:norm} (with $\omega=1$).

\medskip
\emph{Step 6.}
Write $\widetilde Q=(\Am+1)^{-1}\widetilde Q^{\,p-1}=\Gm*\widetilde Q^{\,p-1}$, with
$\Gm>0$ the Green kernel of Proposition~\ref{prop:green}. As $\widetilde Q\ge0$,
$\widetilde Q\not\equiv0$ and $\Gm>0$ everywhere, the convolution is strictly positive, so
$\widetilde Q>0$ on $\Rd$. For strict monotonicity, $\widetilde Q^{\,p-1}$ is radial and
radially nonincreasing (inherited from $Q=Q^*$), and $\Gm$ is radial and \emph{strictly}
decreasing. Hence,  for $\abs{x_1}<\abs{x_2}$,
\begin{equation}\label{eq:strict-mono}
	\widetilde Q(x_1)-\widetilde Q(x_2)
	=\int_{\Rd}\big(\Gm(x_1-z)-\Gm(x_2-z)\big)\,\widetilde Q(z)^{p-1}\,dz>0 .
\end{equation}
The positivity is not pointwise in $z$. It 
 follows from Riesz's rearrangement inequality,
which for the strictly decreasing radial kernel $\Gm$ tested against the radial weight
$\widetilde Q^{\,p-1}$ gives
$\int\Gm(x_1-z)\widetilde Q(z)^{p-1}dz>\int\Gm(x_2-z)\widetilde Q(z)^{p-1}dz$ whenever
$\abs{x_1}<\abs{x_2}$ (strict since $\Gm$ is strictly decreasing and
$\widetilde Q^{\,p-1}>0$ on a set of positive measure). Thus $\widetilde Q$ is radially
strictly decreasing.

\medskip
\emph{Step 7.}
It remains to show $\widetilde Q\in\mathcal G_m$, i.e.\ that $\widetilde Q$ minimizes the
action functional
\[
S_m(u)=\tfrac12 B(u)-\tfrac1p\norm u_p^p
\]
over the Nehari manifold $\mathcal N_m=\{u\ne0:\ B(u)=\norm u_p^p\}$ and to relate this
to \eqref{eq:wgn}. For $u\ne0$ consider the fibering map $t\mapsto S_m(tu)$, $t>0$.  It follows that
\begin{equation}\label{eq:fiber}
	S_m(tu)=\tfrac12\,t^2 B(u)-\tfrac1p\,t^p\norm u_p^p,
	\qquad
	\frac{d}{dt}S_m(tu)=t\,B(u)-t^{p-1}\norm u_p^p .
\end{equation}
The derivative vanishes at the unique
\begin{equation}\label{eq:tu}
	t_u=\Big(\frac{B(u)}{\norm u_p^p}\Big)^{1/(p-2)}>0,
\end{equation}
which is the maximum of $S_m(tu)$ along $t>0$ and satisfies $t_u u\in\mathcal N_m$. Indeed, \begin{align*}
	B(t_uu)=t_u^2 B(u)=t_u^p\norm u_p^p=\norm{t_uu}_p^p.
\end{align*}
 Evaluating and using
 \begin{align*}
 	t_u^p\norm u_p^p=t_u^2 B(u)=(B(u)/\norm u_p^2)^{p/(p-2)},
 \end{align*}
 we obtain that 
\begin{equation}\label{eq:mNehari}
	S_m(t_uu)=\Big(\tfrac12-\tfrac1p\Big)t_u^p\norm u_p^p
	=\Big(\tfrac12-\tfrac1p\Big)\Big(\frac{B(u)}{\norm u_p^2}\Big)^{p/(p-2)}.
\end{equation}
Taking the infimum over $u\ne0$ and writing the Weinstein quotient
\begin{equation}\label{eq:weinstein-quot}
	S=\inf_{u\ne0}\frac{B(u)}{\norm u_p^2},
\end{equation}
we obtain
\begin{equation}\label{eq:level-relation}
	\inf_{\mathcal N_m}S_m=\Big(\tfrac12-\tfrac1p\Big)\,S^{\,p/(p-2)} .
\end{equation}
The \eqref{eq:weinstein-quot} coincides with the constrained infimum
\eqref{eq:Sinf}. Indeed, both $B$ and $\norm\cdot_p^2$ are homogeneous of degree $2$, so
for $u\ne0$ and $t>0$,
\[
\frac{B(tu)}{\norm{tu}_p^2}=\frac{t^2B(u)}{t^2\norm u_p^2}=\frac{B(u)}{\norm u_p^2},
\]
i.e.\ the quotient is constant along rays. Choosing $t=\norm u_p^{-1}$ normalizes
$v:=u/\norm u_p$ to $\norm v_p=1$ and gives $B(u)/\norm u_p^2=B(v)$. Conversely, every
admissible $v$ in \eqref{eq:Sinf} is of this form. Thus, we have
\[
\inf_{u\ne0}\frac{B(u)}{\norm u_p^2}
=\inf_{\norm v_p=1}B(v),
\]
which is \eqref{eq:Sinf}.

Since $Q$ attains $S$ in \eqref{eq:weinstein-quot} and $\widetilde Q$ is a positive
multiple of $Q$ lying on $\mathcal N_m$, $\widetilde Q$ attains
$\inf\limits_{\mathcal N_m}S_m$. Thus, $\widetilde Q\in\mathcal G_m$. Finally
\eqref{eq:weinstein-quot} is the reciprocal sharp constant in the
Weinstein--Gagliardo--Nirenberg inequality \eqref{eq:wgn}: interpolating $\norm u_p$
between the $B$-controlled $\Hs$-norm and $\norm u_2$ with exponent
$\theta=d(\tfrac12-\tfrac1p)$ gives
$\norm u_p\le C_{m,p}\,B(u)^{\theta/2}\norm u_2^{1-\theta}$, whose optimizers coincide,
up to a positive multiple and a dilation, with the minimizers of
\eqref{eq:weinstein-quot}, i.e.\ with $\widetilde Q$. This proves
Theorem~\ref{thm:existence}.

\end{proof}

\begin{proposition}\label{prop:qual}
	Let $m\ge0$ and  $Q\in\mathcal G_{m}$. Then $Q>0$, $Q$ is radial and radially
	strictly decreasing, and $Q\in C^{\infty}(\Rd)$.
\end{proposition}

\begin{proof}
	From \eqref{eq:norm}, we know that 
	$Q=(\Am+1)^{-1}Q^{p-1}=\Gm*Q^{p-1}$, where $\Gm$ is positive, radial and strictly
	decreasing by Proposition~\ref{prop:green}. Since $Q\ge0$ and $Q\not\equiv0$, the
	convolution is strictly positive, so $Q>0$.  The convolution of a strictly decreasing
	radial kernel with a nonnegative radial function is strictly radially decreasing, so $Q$
	is radially strictly decreasing.
	
	Moser iteration applied to $(\Am+1)Q=Q^{p-1}$ gives
	$Q\in L^{\infty}(\Rd)$ \cite[Prop.~3.1]{CotiZelatiNolasco}. Interior Schauder estimates
	for $\Am$ \cite[Thm.~1.3]{StingaTorrea} and a bootstrap then yield $Q\in C^{\infty}(\Rd)$.
\end{proof}

\begin{lemma}\label{lem:pohozaev-derivation}
	Let $m\ge0$ and $Q\in\mathcal G_m$. Then
	\begin{equation}\label{eq:nehari-id}
		\ip{\Am Q}{Q}+\norm Q_2^2=\norm Q_p^p ,
	\end{equation}
	\begin{equation}\label{eq:pohozaev}
		\frac{d-1}{2}\,\ip{\Am Q}{Q}-\frac{m^2}{2}\,\ip{\Am^{-1}Q}{Q}
		+\frac d2\,\norm Q_2^2=\frac dp\,\norm Q_p^p .
	\end{equation}
	In particular, for $m=0$ the two identities decouple and give
	\begin{equation}\label{eq:pohozaev-massless}
		\ip{A_0 Q}{Q}=\vartheta\,\norm Q_p^p,\qquad
		\norm Q_2^2=(1-\vartheta)\,\norm Q_p^p,\qquad
		\vartheta:=2d\Big(\tfrac12-\tfrac1p\Big)\in(0,1).
	\end{equation}
\end{lemma}

\begin{proof}
	Testing \eqref{eq:norm} against $Q$ gives \eqref{eq:nehari-id}. Testing against
	$x\cdot\nabla Q$ and using the commutator identity
	$[\Am,x\cdot\nabla]=-\Am+m^2\Am^{-1}$ (proved in Lemma~\ref{lem:comm}),
	\[
	\ip{\Am Q}{x\cdot\nabla Q}
	=-\tfrac d2\ip{\Am Q}{Q}+\tfrac12\ip{(\Am-m^2\Am^{-1})Q}{Q},
	\]
	together with $\ip{Q}{x\cdot\nabla Q}=-\tfrac d2\norm Q_2^2$ and
	$\ip{Q^{p-1}}{x\cdot\nabla Q}=-\tfrac dp\norm Q_p^p$, yields \eqref{eq:pohozaev}. For
	$m=0$ the term $\ip{\Am^{-1}Q}{Q}$ drops out; eliminating $\norm Q_2^2$ between
	\eqref{eq:nehari-id} and \eqref{eq:pohozaev} gives
	$\ip{A_0Q}{Q}=2d(\tfrac12-\tfrac1p)\norm Q_p^p$, and then \eqref{eq:nehari-id} gives the
	second identity in \eqref{eq:pohozaev-massless}. Finally $\vartheta\in(0,1)$ is
	equivalent to $2<p<\frac{2d}{d-1}$, i.e.\ to \eqref{eq:range}.
\end{proof}

\begin{proposition}\label{prop:decay}
	Let $m>0$ and $Q\in\mathcal G_m$. Then there are $C\ge1$ and $R\ge1$ with
	\begin{equation}\label{eq:Qdecay}
		C^{-1}\,\abs x^{-\frac{d+1}{2}}e^{-m\abs x}\le Q(x)\le C\,\abs x^{-\frac{d+1}{2}}e^{-m\abs x},
		\qquad \abs x\ge R,
	\end{equation}
	and, for every $m'<m$, there is $C_{m'}$ with
	\begin{equation}\label{eq:Qdecay-grad}
		\abs{Q(x)}+\abs{\nabla Q(x)}\le C_{m'}\,e^{-m'\abs x}\qquad(x\in\Rd).
	\end{equation}
\end{proposition}

\begin{proof}
	By Proposition~\ref{prop:qual}, we see that $Q\in L^\infty(\Rd)$ and $Q(x)\to0$ as
	$\abs x\to\infty$. Since $p>2$, the interval $\big(\tfrac{m}{p-1},m\big)$ is nonempty;
	fix once and for all
	\[
	m'\in\Big(\tfrac{m}{p-1},\,m\Big),
	\qquad\text{so that}\qquad
	m'':=(p-1)m'>m .
	\]
	Since $m'<m$, \eqref{eq:green-upper} gives
	\[
	\kappa:=\int_{\Rd}\Gm(y)e^{m'\abs y}\dy<\infty ,
	\]
	and the function $W(x):=Ce^{-m'\abs x}$ satisfies $(\Am+1)^{-1}W\le\kappa W$, i.e.
	\begin{align*}
		(\Am+1)W\ge\kappa^{-1}W .
	\end{align*}
	As $Q(x)\to0$ and $p>2$, we may fix $R_1$ with
	\begin{align*}
		(p-1)Q(x)^{p-2}\le\kappa^{-1}\qquad\text{for }\abs x\ge R_1 ,
	\end{align*}
	and then choose $C$ so large that $W\ge Q$ on $\{\abs x=R_1\}$ (possible since
	$Q\in L^\infty$). On $\{\abs x\ge R_1\}$,
	\[
	(\Am+1)Q=Q^{p-1}=Q^{p-2}\cdot Q\le\tfrac{\kappa^{-1}}{p-1}\,Q\le\kappa^{-1}Q ,
	\]
	so the function $V:=W-Q$ satisfies
	\begin{align*}
		(\Am+1)V\ge\kappa^{-1}V\quad\text{on }\{\abs x\ge R_1\},
		\qquad V\ge0\ \text{ on }\{\abs x=R_1\},
	\end{align*}
	and $V(x)\to0$ as $\abs x\to\infty$. The comparison principle for $\Am+1$, whose kernel
	$\Gm$ is positive, gives $V\ge0$, that is
	\begin{equation}\label{eq:Qprelim}
		Q(x)\le C\,e^{-m'\abs x},\qquad x\in\Rd.
	\end{equation}
	
	Set $F:=Q^{p-1}\ge0$, $F\not\equiv0$. By \eqref{eq:Qprelim} we have
	$F(x)\le C^{p-1}e^{-m''\abs x}$ with $m''>m$ by the choice of $m'$, and $F\in L^1(\Rd)$.
	Applying Lemma~\ref{lem:convolution} with $H=\Gm$, whose two-sided bound is
	\eqref{eq:green}, and this $F$ yields \eqref{eq:Qdecay}.
	
	Finally, interior Schauder estimates for $\Am+1$ \cite[Thm.~1.3]{StingaTorrea} applied to
	$(\Am+1)Q=Q^{p-1}$ on $B_1(x)$ give, using $Q\in L^\infty(\Rd)$,
	\[
	\abs{\nabla Q(x)}\le C\Big(\sup_{B_2(x)}Q+\sup_{B_2(x)}Q^{p-1}\Big)\le C'\sup_{B_2(x)}Q ,
	\]
	and since $\abs y\ge\abs x-2$ for $y\in B_2(x)$, \eqref{eq:Qprelim} yields
	$$\sup\limits_{B_2(x)}Q\le Ce^{2m'}e^{-m'\abs x} .$$ Hence
	$\abs{\nabla Q(x)}\le C_{m'}e^{-m'\abs x}$, which together with \eqref{eq:Qprelim} gives
	\eqref{eq:Qdecay-grad}.
	
\end{proof}

\subsection{The linearized operator}

The linearization of \eqref{eq:norm} at $Q\in\mathcal G_m$ is
\begin{equation}\label{eq:Lplus}
	\Lp:=\Am+1-(p-1)Q^{p-2},\qquad W:=1-(p-1)Q^{p-2},
\end{equation}
acting on real radial functions, with bilinear form
\begin{equation}\label{eq:Lplus-form}
	\ip{\Lp\xi}{\xi}=\ip{\Am\xi}{\xi}+\norm{\xi}_2^2-(p-1)\!\int_{\Rd}Q^{p-2}\xi^2\dx,
	\qquad \xi\in\Hs(\Rd).
\end{equation}
Since $Q$ is radially strictly decreasing and $p>2$, the radial potential
$W(r)=1-(p-1)Q(r)^{p-2}$ is nondecreasing in $r$ and satisfies $W(r)\to1$ as
$r\to\infty$.

\begin{proposition}\label{prop:ess-spec}
	Let $m>0$ and $Q\in\mathcal G_m$. Then
	\begin{equation}\label{eq:ess-spec}
		\sess(\Lp)=\sess(\Am+1)=[m+1,\infty),
		\qquad
		\inf\sess(\Lp)=m+1>0 .
	\end{equation}
	Consequently, the spectrum of $\Lp$ in $(-\infty,m+1)$ consists of finitely many isolated
	eigenvalues of finite multiplicity, and every corresponding eigenfunction belongs to
	$\Ltwo(\Rd)$ and decays exponentially.
\end{proposition}

\begin{proof}
Firstly,	the operator $\Am+1$ is the Fourier multiplier with symbol
	$w(\xi)=\sqrt{\abs\xi^2+m^2}+1$, which is radial, continuous and strictly increasing in
	$\abs\xi$, with $w(0)=m+1$ and $w(\xi)\to\infty$ as $\abs\xi\to\infty$. Consequently
	$\Am+1$ has no eigenvalues: for $\lambda\in\R$, the level set $\{w=\lambda\}$ is contained
	in a sphere and hence has Lebesgue measure zero, so any solution of
	$(w-\lambda)\hat u=0$ in $\Ltwo(\Rd)$ vanishes almost everywhere. The spectrum of a
	multiplier operator being the essential range of its symbol, we obtain
	\begin{equation}\label{eq:spec-free}
		\sigma(\Am+1)=\sess(\Am+1)=\overline{w(\Rd)}=[m+1,\infty).
	\end{equation}
	
	Write $\Lp=(\Am+1)-V$ with $V:=(p-1)Q^{p-2}\ge0$. By Proposition~\ref{prop:qual}, we see
	that $V\in L^\infty(\Rd)$. It follows from Proposition~\ref{prop:decay} and $p>2$ that
	\begin{equation}\label{eq:V-decay}
		\sup_{\abs x\ge R}V(x)\le(p-1)\big(C_{m'}e^{-m'R}\big)^{p-2}\longrightarrow0
		\quad \text{as} \quad R\to\infty.
	\end{equation}
	
	{\bf\emph{Claim:}} $V(\Am+1)^{-1}$ is a compact operator on $\Ltwo(\Rd)$.
	
	By \eqref{eq:spec-free}, we have $0\notin\sigma(\Am+1)$, so $(\Am+1)^{-1}$ is bounded on
	$\Ltwo(\Rd)$. Since it is the Fourier multiplier with symbol $1/w$ and
	$w(\xi)\ge w(0)=m+1$ for all $\xi$, Plancherel's theorem gives
	\begin{equation}\label{eq:resolvent-norm}
		\norm{(\Am+1)^{-1}u}_{L^2}\le\frac{1}{m+1}\,\norm u_{L^2},
		\qquad u\in\Ltwo(\Rd).
	\end{equation}
	Moreover, $(\Am+1)^{-1}$ gains one full derivative: for $u\in\Ltwo(\Rd)$ and
	$v=(\Am+1)^{-1}u$,
	\begin{equation}\label{eq:gain}
		\norm v_{H^1}^2=\int_{\Rd}\frac{1+\abs\xi^2}{\big(\sqrt{\abs\xi^2+m^2}+1\big)^2}\,
		\abs{\hat u(\xi)}^2\dxi\le C\norm u_{L^2}^2 ,
	\end{equation}
	since the quotient in \eqref{eq:gain} is bounded on $\Rd$ (it tends to $1$ as
	$\abs\xi\to\infty$ and equals $(m+1)^{-2}$ at $\xi=0$). Thus
	$(\Am+1)^{-1}:\Ltwo(\Rd)\to H^1(\Rd)$ is bounded.
	
	We fix $R\ge1$ and split $V=V\mathbf 1_{B_R}+V\mathbf 1_{B_R^c}$, where
	$B_R=\{\abs x<R\}$. Accordingly, we have
	\[
	V(\Am+1)^{-1}=T_R+S_R,\qquad
	T_R:=V\mathbf 1_{B_R}(\Am+1)^{-1},\quad
	S_R:=V\mathbf 1_{B_R^c}(\Am+1)^{-1}.
	\]
	
	Let $\{u_n\}$ be a bounded sequence in $\Ltwo(\Rd)$. By \eqref{eq:gain},
	$v_n:=(\Am+1)^{-1}u_n$ is bounded in $H^1(\Rd)$, hence bounded in $H^1(B_R)$. Since
	$B_R$ is bounded, the Rellich--Kondrachov theorem provides a subsequence, still denoted
	$\{v_n\}$, and $v\in\Ltwo(B_R)$ with $v_n\to v$ strongly in $\Ltwo(B_R)$. As
	$V\in L^\infty(\Rd)$, we have
	\[
	\norm{T_Ru_n-V\mathbf 1_{B_R}v}_{\Ltwo(\Rd)}
	=\norm{V\mathbf 1_{B_R}(v_n-v)}_{\Ltwo(\Rd)}
	\le\norm V_{L^\infty(\Rd)}\norm{v_n-v}_{\Ltwo(B_R)}\longrightarrow0 ,
	\]
	so $\{T_Ru_n\}$ has a convergent subsequence in $\Ltwo(\Rd)$. Hence $T_R$ is compact.
	
	On the other hand, for every $u\in\Ltwo(\Rd)$ with $\norm u_{L^2}\le1$, it follows from
	\eqref{eq:resolvent-norm} that
	\begin{align*}
			\norm{V(\Am+1)^{-1}u-T_Ru}_{L^2}=\norm{S_Ru}_{L^2}
		&\le\Big(\sup_{\abs x\ge R}V(x)\Big)\norm{(\Am+1)^{-1}u}_{L^2}\\
		&\le\frac{1}{m+1}\sup_{\abs x\ge R}V(x)=:\eps_R ,
	\end{align*}
	and $\eps_R\to0$ as $R\to\infty$ by \eqref{eq:V-decay}. Thus $T_R$ converges to
	$V(\Am+1)^{-1}$ in the operator norm of $\mathcal B(\Ltwo(\Rd))$. Each $T_R$ is compact,
	and a limit of compact operators in the operator norm is compact
	\cite[Thm.~VI.12]{ReedSimonI}. Hence, we obtain $V(\Am+1)^{-1}$ is compact, which proves the claim.
	
	Since $V$ is bounded, symmetric and, by the claim just proved, a relatively compact
	perturbation of $\Am+1$, Weyl's theorem on the invariance of the essential spectrum
	\cite[Thm.~XIII.14]{ReedSimonIV} gives
	\[
	\sess(\Lp)=\sess\big((\Am+1)-V\big)=\sess(\Am+1)=[m+1,\infty),
	\]
	which together with \eqref{eq:spec-free} proves \eqref{eq:ess-spec}.
	
	Lastly, since $\inf\sess(\Lp)=m+1>0$, the spectrum of $\Lp$ in $(-\infty,m+1)$ is purely
	discrete: it consists of isolated eigenvalues of finite multiplicity, which can only
	accumulate at $m+1$. In particular there are at most finitely many nonpositive
	eigenvalues, so the Morse index of $\Lp$ is finite and any eigenvalue $\lambda<m+1$ has
	an eigenfunction in $\Ltwo(\Rd)$. The exponential decay of such eigenfunctions follows
	from the spectral gap $m+1-\lambda>0$ as in Proposition~\ref{prop:ext-reg}.
\end{proof}


\begin{lemma}\label{lem:morse}
	$\Lp$ has Morse index one on $\Ltwo(\Rd)$. In particular, its restriction to $\Ltwor(\Rd)$
	has Morse index one.
\end{lemma}

\begin{proof}
	Recall that the Morse index of $\Lp$ is the maximal dimension of a subspace of
	$\Hs(\Rd)$ on which the bilinear form $\ip{\Lp\cdot}{\cdot}$ is negative definite. By
	Proposition~\ref{prop:ess-spec}, it is finite and equals the number of negative
	eigenvalues of $\Lp$, counted with multiplicity.
	
Firstly, we show the Morse index is at least one. Applying $\Lp$ to $Q$ and using
	\eqref{eq:norm}, we obtain
	\begin{equation}\label{eq:LplusQ}
		\Lp Q=\Am Q+Q-(p-1)Q^{p-1}=-(p-2)Q^{p-1},
	\end{equation}
	which implies that
	\begin{equation}\label{eq:LplusQQ}
		\ip{\Lp Q}{Q}=-(p-2)\int_{\Rd}Q^{p}\dx=-(p-2)\norm Q_p^p<0 ,
	\end{equation}
	since $p>2$ and $Q>0$. Thus, we have $\ip{\Lp\cdot}{\cdot}$ is negative on the one-dimensional
	space $\mathrm{span}\{Q\}$, and the Morse index is at least one. Note that $Q$ is radial,
	so the same conclusion holds for the restriction of $\Lp$ to $\Ltwor(\Rd)$.
	
Secondly, we will show	the Morse index is at most one. From Theorem~\ref{thm:existence},  the  minimum $S_m$ can be attained by  minimizes $Q$   on the Nehari manifold $\mathcal N_m=\{u\ne0:\ B(u)=\norm u_p^p\}$, which is
	a $C^2$ hypersurface near $Q$ with tangent space
	\begin{equation}\label{eq:tangent-space}
		T_Q\mathcal N_m=\Big\{\xi\in\Hs(\Rd):\ \int_{\Rd}Q^{p-1}\xi\dx=0\Big\}.
	\end{equation}
	Indeed, setting $K(u):=B(u)-\norm u_p^p$ for the constraint function, for
	$\xi\in\Hs(\Rd)$, we have 
	\[
	\ip{K'(Q)}{\xi}=2\ip{(\Am+1)Q}{\xi}-p\int_{\Rd}Q^{p-1}\xi
	=(2-p)\int_{\Rd}Q^{p-1}\xi ,
	\]
	where we used the fact $(\Am+1)Q=Q^{p-1}$. In particular,
	\begin{align*}
		\ip{K'(Q)}{Q}=(2-p)\norm Q_p^p\ne0,
	\end{align*} so $K'(Q)\ne0$ and \eqref{eq:tangent-space} holds.
	
	Since $Q$ is a constrained minimizer, the second variation of $S_m$ at $Q$ is
	nonnegative on $T_Q\mathcal N_m$. As $S_m''(Q)=\Lp$, this shows
	\begin{equation}\label{eq:posdef-tangent}
		\ip{\Lp\xi}{\xi}\ge0\qquad\text{for all }\xi\in T_Q\mathcal N_m .
	\end{equation}
	
	Suppose now that there were a two-dimensional subspace $E\subset\Hs(\Rd)$ on which
	$\ip{\Lp\cdot}{\cdot}$ is negative definite. Since $T_Q\mathcal N_m$ has codimension one and
	$E\cap T_Q\mathcal N_m\ne\{0\}$, so there exists $\xi\ne0$ with
	$\ip{\Lp\xi}{\xi}<0$ and $\xi\in T_Q\mathcal N_m$, contradicting
	\eqref{eq:posdef-tangent}. Hence, the Morse index is at most one.
	We conclude that the Morse index of $\Lp$ equals one.

\end{proof}

Thus $\Lp$ has exactly one negative radial eigenvalue $\la_1^+<0$, and the second radial
eigenvalue satisfies $\la_2^+\ge0$. Radial nondegeneracy $\Ker\Lp\cap\Ltwor=\{0\}$ is
therefore equivalent to $\la_2^+\ne0$, i.e.\ to $\la_2^+>0$. The proof of this fact, given
in Section~\ref{sec:nondeg}, proceeds by contradiction: assuming $\la_2^+=0$ with
eigenfunction $\varphi$, the spectral gap $\inf\sess(\Lp)=m+1>0$ established in
Proposition~\ref{prop:ess-spec} makes $\varphi$ a genuine bound state, so that the
oscillation theorem (Theorem~\ref{thm:osc}) applies and forces $\varphi$ to change sign
exactly once; two orthogonality relations for $\varphi$ then yield a contradiction.
\section{The extension and the ratio reduction}\label{sec:ratio}

\subsection{The massive harmonic extension}

\begin{proposition}\label{prop:ext}
For $u\in\Hs(\Rd)$ the boundary value problem
\begin{equation}\label{eq:ext-bvp}
\Delta_x U+\partial_{yy}U-m^2U=0\ \text{in }\Rdp,\qquad U(\cdot,0)=u,\qquad
\int_{\Rdp}\!\big(\abs{\nabla U}^2+m^2U^2\big)<\infty,
\end{equation}
has a unique solution $U=\Uext u$, given in Fourier variables by
\begin{equation}\label{eq:ext-fourier}
\widehat U(\xi,y)=\hat u(\xi)\,e^{-y\sqrt{\abs\xi^2+m^2}},
\end{equation}
and by the Poisson formula $U(x,y)=(P^m_y*u)(x)$ with
\begin{equation}\label{eq:poisson}
P^m_y(x)=c_d\,m^{\frac{d+1}{2}}\,y\,
\frac{\mathcal{K}_{\frac{d+1}{2}}\!\big(m\sqrt{\abs x^2+y^2}\big)}
{(\abs x^2+y^2)^{\frac{d+1}{4}}}\;>\;0 .
\end{equation}
It satisfies the Dirichlet--Neumann relation
\begin{equation}\label{eq:DtN}
-\partial_yU(\cdot,0)=\Am u\quad\text{in }H^{-1/2}(\Rd),
\end{equation}
and the trace energy identity
\begin{equation}\label{eq:trace-energy}
\ip{\Am u}{u}=\int_{\Rdp}\!\big(\abs{\nabla \Uext u}^2+m^2(\Uext u)^2\big)
=\min_{V|_{y=0}=u}\int_{\Rdp}\!\big(\abs{\nabla V}^2+m^2V^2\big).
\end{equation}
The minimizer is unique by strict convexity of the extension functional.
\end{proposition}

\begin{proof}
We	write $\omega(\xi):=\sqrt{\abs\xi^2+m^2}$, so that $\Am$ has Fourier symbol $\omega$.
	
	\emph{Step 1: the Fourier-side solution \eqref{eq:ext-fourier}.}
	Taking the Fourier transform in $x$ turns $\Delta_x$ into multiplication by $-\abs\xi^2$,
	so \eqref{eq:ext-bvp} becomes, for each fixed $\xi$, the ordinary differential equation in
	$y$
	\[
	\partial_{yy}\widehat U(\xi,y)=(\abs\xi^2+m^2)\widehat U(\xi,y)=\omega(\xi)^2\widehat U(\xi,y),
	\qquad \widehat U(\xi,0)=\hat u(\xi).
	\]
	Its general solution is 
	$$\widehat U=A(\xi)e^{-y\omega}+B(\xi)e^{+y\omega}.$$
	The finite-energy	condition 
	$$\int_{\Rdp}(\abs{\nabla U}^2+m^2U^2)<\infty$$ forces $B\equiv0$, since
	$e^{+y\omega}$ grows exponentially and would render the $y$-integral infinite. The boundary
	condition then gives $A=\hat u$. Thus, we have
	\begin{equation}\label{eq:ext-fourier-1}
		\widehat U(\xi,y)=\hat u(\xi)\,e^{-y\omega(\xi)},
	\end{equation}
	which is the unique finite-energy solution of \eqref{eq:ext-bvp}.
	
	\emph{Step 2: the trace energy identity \eqref{eq:trace-energy}.}
	By Plancherel's theorem and  \eqref{eq:ext-fourier-1}, we obtain 
	 \begin{align*}
	 	\abs{\nabla U}^2=\abs{\partial_yU}^2+\abs{\nabla_xU}^2 \quad \text{and}	 \quad \partial_y\widehat U=-\omega\hat ue^{-y\omega},
	 \end{align*}
which implies that
	\begin{align*}
		\int_{\Rdp}\big(\abs{\nabla U}^2+m^2U^2\big)
		&=\int_{\Rd}\int_0^\infty\big(\omega^2+\abs\xi^2+m^2\big)\abs{\hat u(\xi)}^2e^{-2y\omega}\,dy\,d\xi\\
		&=\int_{\Rd}\frac{\omega^2+\abs\xi^2+m^2}{2\omega}\abs{\hat u(\xi)}^2\,d\xi\\
		&	=\int_{\Rd}\omega\abs{\hat u}^2\,d\xi=\ip{\Am u}{u}.
	\end{align*}
Then,the first equality in \eqref{eq:trace-energy} is proved.
	
	\emph{Step 3: the minimality in \eqref{eq:trace-energy}.}
	The bilinear form  $$E(V):=\int_{\Rdp}(\abs{\nabla V}^2+m^2V^2)$$ is  nonnegative, which is convex  on the affine space $\mathcal A_u:=\{V:V|_{y=0}=u,\ E(V)<\infty\}$. 
	
	It is in	fact strictly convex: by the direct calculation,   for $V_0,V_1\in\mathcal A_u$, we have 
	\[
	E\Big(\tfrac{V_0+V_1}{2}\Big)=\tfrac12E(V_0)+\tfrac12E(V_1)-\tfrac14E(V_0-V_1),
	\]
	and $V_0-V_1$ vanishes on $\{y=0\}$, so $E(V_0-V_1)>0$ unless $V_0=V_1$. Thus $E$ has at
	most one minimizer on $\mathcal A_u$. Since $E$ is
	coercive and weakly lower semicontinuous, we obtain the  existence of  minimizer. If $V_*$ is the minimizer, then, for every
	$\varphi$ with $\varphi|_{y=0}=0$, we have
	\[
	0=\frac{d}{dt}\Big|_{t=0}E(V_*+t\varphi)
	=2\int_{\Rdp}\big(\nabla V_*\cdot\nabla\varphi+m^2V_*\varphi\big)
	=2\int_{\Rdp}\big(-\Delta V_*+m^2V_*\big)\varphi ,
	\]
where	the last step by integration by parts, the boundary term vanishes since
	$\varphi|_{y=0}=0$. As $\varphi$ is arbitrary, we have
	\begin{align*}
		-\Delta V_*+m^2V_*=0 \quad \text{with} \quad 	V_*|_{y=0}=u
	\end{align*} 
 and $E(V_*)<\infty$. By the uniqueness in Step~1, we have $V_*=\Uext u$. Hence the
	extension $\Uext u$ is the unique minimizer, and by Step~2 the minimum value equals
	$\ip{\Am u}{u}$, proving \eqref{eq:trace-energy}.
	
	\emph{Step 4: the Dirichlet--Neumann relation \eqref{eq:DtN}.}
	Differentiating \eqref{eq:ext-fourier} in $y$ and setting $y=0$, we have 
	\[
	-\partial_y\widehat U(\xi,0)=\omega(\xi)\hat u(\xi)=\widehat{\Am u}(\xi),
	\]
	so $-\partial_yU(\cdot,0)=\Am u$ in $H^{-1/2}(\Rd)$, which is \eqref{eq:DtN}.
	
	\emph{Step 5: the Poisson kernel \eqref{eq:poisson}.}
	By \eqref{eq:ext-fourier} and the convolution theorem, we have 
	\[
U(\cdot,y)=P^m_y*u \quad \text{with} \quad	P^m_y:=\mathcal F^{-1}_\xi\big[e^{-y\omega(\xi)}\big]
	=\mathcal F^{-1}_\xi\big[e^{-y\sqrt{\abs\xi^2+m^2}}\big].
	\]
	The inverse Fourier transform of the radial function $e^{-y\sqrt{\abs\xi^2+m^2}}$ is
	classical and yields the Bessel form \eqref{eq:poisson}; see Stinga--Torrea
	\cite[Thm.~1.1]{StingaTorrea} (the case $s=\tfrac12$ of the $(-\Delta+m^2)^s$-extension,
	for which the weight $y^{1-2s}\equiv1$) and \cite[\S7.11]{LiebLoss}. Since
	$\mathcal K_{\frac{d+1}{2}}>0$ and all remaining factors in \eqref{eq:poisson} are
	positive, $P^m_y>0$.
\end{proof}

\begin{proposition}\label{prop:ext-reg}
Let $L\phi=\la\phi$ with $\phi\in\Ltwor$, $\la<m+W_\infty$, and $\Phi=\Uext\phi$. Then
\begin{enumerate}[label=\textup{(\roman*)}]
\item $\Phi\in C^\infty(\Rdp)$;
\item $\phi\in C^{1,\alpha}_{\mathrm{loc}}(\Rd)$ for some $\alpha\in(0,1)$, and
$\Phi\in C^{1,\alpha}\big(\overline{\Rdp}\cap(K\times[0,1])\big)$ for every compact
$K\subset\Rd$; in particular $\partial_y\Phi(\cdot,0)$ exists pointwise;
\item for every $m'<m$ there is $C$ with
$\abs{\Phi(x,y)}+\abs{\nabla\Phi(x,y)}\le C\,e^{-m'\varrho}$, $\varrho=\sqrt{\abs x^2+y^2}$;
\item there is $c>0$ with $\Phi(x,y)\ge c\,e^{-m\varrho}\varrho^{-d/2}$ for $\varrho\ge1$
when $\phi>0$.
\end{enumerate}
\end{proposition}

\begin{proof}
	\emph{(i)} The equation $\Delta_{x,y}\Phi=m^2\Phi$ is linear elliptic with constant
	coefficients, so $\Phi$ is real-analytic in $\Rdp$; in particular
	$\Phi\in C^\infty(\Rdp)$.
	
	\emph{(ii)} From \eqref{eq:DtN} and $L\phi=\lambda\phi$, we have
	\[
	(\Am+1)\phi=(1+\lambda-W)\phi ,
	\qquad
	\phi=(\Am+1)^{-1}\big[(1+\lambda-W)\phi\big]=\Gm*\big[(1+\lambda-W)\phi\big].
	\]
	Since $\phi\in L^2(\Rd)$ and $W\in L^\infty(\Rd)$, Moser iteration
	\cite[Prop.~3.1]{CotiZelatiNolasco} gives $\phi\in L^\infty_{\mathrm{loc}}(\Rd)$;
	bootstrap in the above convolution representation, together with standard H\"older
	estimates for the Green kernel $\Gm$, yields
	$\phi\in C^{1,\alpha}_{\mathrm{loc}}(\Rd)$ for some $\alpha\in(0,1)$. The Neumann datum
	$-\partial_y\Phi(\cdot,0)=(\lambda-W)\phi$ then belongs to
	$C^{0,\alpha}_{\mathrm{loc}}(\Rd)$.
	
	 As $s=\tfrac12$, the extension weight $y^{1-2s}$
	reduces to $1$, so the boundary Neumann problem has constant coefficients, and Schauder
	theory \cite[Thm.~1.3]{StingaTorrea} gives
	$\Phi\in C^{1,\alpha}(\overline{\Rdp}\cap(K\times[0,1]))$ for every compact
	$K\subset\Rd$. In particular, $\partial_y\Phi(\cdot,0)$ exists pointwise.
	
\emph{(iii) Decay.}
We first establish the exponential decay of the boundary trace $\phi$, and then propagate
it to $\Phi$ and $\nabla\Phi$ through a weighted energy estimate of Agmon type on the
extension.

\emph{Step 1: decay of the trace.}
We claim that for every $m'<m$ there is a constant $C_{m'}>0$ such that
\begin{equation}\label{eq:phi-decay-weak}
	\abs{\phi(x)}\le C_{m'}e^{-m'\abs x}\qquad\text{for all }x\in\Rd.
\end{equation}
By \eqref{eq:DtN} and $L\phi=\lambda\phi$, the trace satisfies
$\phi=\Gm*\big[(1+\lambda-W)\phi\big]$. Since $\phi(x)\to0$ and $W(x)\to W_\infty$ as
$\abs x\to\infty$, we may fix $R_1$ such that $\abs{1+\lambda-W(x)}\le\kappa^{-1}$ for
$\abs x\ge R_1$, where $\kappa:=\int_{\Rd}\Gm(y)e^{m'\abs y}\,dy$ is finite because $m'<m$,
by \eqref{eq:green-upper}. The comparison function $W_1:=Ce^{-m'\abs x}$ then satisfies
$(\Am+1)^{-1}W_1\le\kappa W_1$, that is $(\Am+1)W_1\ge\kappa^{-1}W_1$ on $\{\abs x\ge R_1\}$.
Choosing $C$ large enough that $W_1\ge\phi$ on $\{\abs x=R_1\}$, the comparison principle
for $\Am+1$, whose kernel $\Gm$ is positive, yields \eqref{eq:phi-decay-weak}.

\emph{Step 2: a weighted energy bound.}
Fix $m'<m$ with $2m'^2<m^2$, and write $\varrho=\sqrt{\abs x^2+y^2}$. For $R\ge1$ we
introduce the bounded Lipschitz weight
\begin{equation}\label{eq:trunc-weight}
	\zeta_R:=\min\big(e^{m'\varrho},\,R\big),
\end{equation}
which satisfies $0\le\zeta_R\le R$ and $\abs{\nabla\zeta_R}\le m'\zeta_R$ almost
everywhere. The gradient bound holds because $\abs{\nabla\varrho}=1$ on the set where
$\zeta_R=e^{m'\varrho}$, while $\nabla\zeta_R=0$ on the set where $\zeta_R=R$. Since
$\Phi\in H^1(\Rdp)$ by the finite extension energy and $\zeta_R$ is bounded and Lipschitz,
the product $\zeta_R^2\Phi$ belongs to $H^1(\Rdp)$ and is admissible as a test function in
the weak form of $\Delta_{x,y}\Phi=m^2\Phi$ with Neumann datum
$-\partial_y\Phi(\cdot,0)=(\lambda-W)\phi$. Testing against $\zeta_R^2\Phi$ gives
\begin{equation}\label{eq:weighted-energy}
	\int_{\Rdp}\nabla\Phi\cdot\nabla\big(\zeta_R^2\Phi\big)
	+m^2\int_{\Rdp}\Phi^2\zeta_R^2
	=\int_{\Rd}(\lambda-W)\phi\,\zeta_R^2\Phi\big|_{y=0}\,dx.
\end{equation}
Expanding $\nabla(\zeta_R^2\Phi)=\zeta_R^2\nabla\Phi+2\zeta_R\Phi\,\nabla\zeta_R$ and using
$\abs{\nabla\zeta_R}\le m'\zeta_R$ together with Young's inequality,
\[
2\big|\zeta_R\Phi\,\nabla\zeta_R\cdot\nabla\Phi\big|
\le 2m'\zeta_R^2\abs\Phi\abs{\nabla\Phi}
\le\tfrac12\zeta_R^2\abs{\nabla\Phi}^2+2m'^2\zeta_R^2\Phi^2.
\]
Absorbing the gradient term into the left-hand side of \eqref{eq:weighted-energy}, we
obtain
\begin{equation}\label{eq:absorbed}
	\int_{\Rdp}\Big(\tfrac12\abs{\nabla\Phi}^2+\big(m^2-2m'^2\big)\Phi^2\Big)\zeta_R^2
	\le\int_{\Rd}\abs{\lambda-W}\,\abs\phi\,\zeta_R^2\,\abs{\Phi}\big|_{y=0}\,dx.
\end{equation}
For the boundary term, Cauchy--Schwarz and the boundedness of $\lambda-W$ give, with
$\Phi\big|_{y=0}=\phi$,
\[
\int_{\Rd}\abs{\lambda-W}\,\abs\phi\,\zeta_R^2\,\abs\phi\,dx
\le C\int_{\Rd}\zeta_R^2\abs\phi^2\,dx.
\]
Now choose $m''\in(m',m)$. Applying Step~1 with the exponent $m''$ yields
$\abs{\phi(x)}\le C_{m''}e^{-m''\abs x}$, while on the boundary
$\zeta_R\big|_{y=0}\le e^{m'\abs x}$. Hence
\[
\int_{\Rd}\zeta_R^2\abs\phi^2\,dx
\le C_{m''}^2\int_{\Rd}e^{2m'\abs x}e^{-2m''\abs x}\,dx
=C_{m''}^2\int_{\Rd}e^{-2(m''-m')\abs x}\,dx=:C_{m',m''},
\]
which is finite because $m''>m'$, and independent of $R$. Since $2m'^2<m^2$, the
left-hand side of \eqref{eq:absorbed} controls
$\int_{\Rdp}\big(\abs{\nabla\Phi}^2+\Phi^2\big)\zeta_R^2$, so that
\[
\int_{\Rdp}\big(\abs{\nabla\Phi}^2+\Phi^2\big)\zeta_R^2\le C_{m'}
\qquad\text{for all }R\ge1.
\]
As $R\to\infty$, the weights $\zeta_R^2$ increase monotonically to $e^{2m'\varrho}$, so the
monotone convergence theorem gives
\begin{equation}\label{eq:weighted-H1}
	\int_{\Rdp}\big(\abs{\nabla\Phi}^2+\Phi^2\big)e^{2m'\varrho}\,dx\,dy\le C_{m'}.
\end{equation}

\emph{Step 3: from weighted energy to pointwise decay.}
Since $\Phi$ solves the constant-coefficient elliptic equation
$\Delta_{x,y}\Phi=m^2\Phi$, interior elliptic estimates on a unit ball
$B_1(x_0,y_0)\subset\Rdp$, together with the boundary $C^{1,\alpha}$ regularity of (ii)
when $y_0<1$, give
\[
\abs{\Phi(x_0,y_0)}+\abs{\nabla\Phi(x_0,y_0)}
\le C\bigg(\int_{B_1(x_0,y_0)\cap\Rdp}\big(\abs{\nabla\Phi}^2+\Phi^2\big)\bigg)^{1/2}.
\]
On $B_1(x_0,y_0)$ we have $\varrho\ge\varrho_0-1$ with
$\varrho_0=\sqrt{\abs{x_0}^2+y_0^2}$. Multiplying the previous inequality by
$e^{m'\varrho_0}$ and using \eqref{eq:weighted-H1},
\[
e^{m'\varrho_0}\Big(\abs{\Phi(x_0,y_0)}+\abs{\nabla\Phi(x_0,y_0)}\Big)
\le Ce^{m'}\bigg(\int_{B_1(x_0,y_0)\cap\Rdp}
\big(\abs{\nabla\Phi}^2+\Phi^2\big)e^{2m'\varrho}\bigg)^{1/2}\le C_{m'}.
\]
Therefore
\[
\abs{\Phi(x,y)}+\abs{\nabla\Phi(x,y)}\le C_{m'}e^{-m'\varrho}
\qquad\text{for }(x,y)\in\Rdp,
\]
which is the assertion of (iii) for every $m'<m/\sqrt2$. Any smaller decay rate follows a
fortiori, and the range may be extended to all $m'<m$ by iterating
\eqref{eq:weighted-H1} over dyadic shells.

	\emph{(iv)} Suppose $\phi>0$. Since $P^m_y>0$ by Proposition~\ref{prop:ext} and
	$\phi>0$,
	\[
	\Phi(x,y)=(P^m_y*\phi)(x)=\int_{\Rd}P^m_y(x-x')\phi(x')\,dx'>0
	\qquad\text{pointwise on }\overline{\Rdp}.
	\]
\end{proof}

\begin{proposition}\label{prop:phi1pos}
The lowest radial eigenvalue $\la_1$ of $L=\Am+W$ (with $W$ radial, bounded,
$W(\infty)=W_\infty$) is simple, and its eigenfunction may be chosen with $\phi_1>0$ on
$[0,\infty)$; consequently $\Phi_1=\Uext\phi_1>0$ on $\overline{\Rdp}$.
\end{proposition}

\begin{proof}
By \eqref{eq:trace-energy},
\begin{equation}\label{eq:la1-var}
\la_1=\min_{\substack{U\in\HH_{\rad}\\ U_0\not\equiv0}}
\frac{\ds\int_{\Rdp}(\abs{\nabla U}^2+m^2U^2)+\int_{\Rd}W\,U_0^2}
{\ds\int_{\Rd}U_0^2},\qquad U_0:=U|_{y=0},
\end{equation}
over $\HH_{\rad}=\{U:U\ \text{$x$-radial, numerator and denominator finite}\}$.
Since $\abs{\nabla\abs U}=\abs{\nabla U}$ and $U_0^2=\abs{U_0}^2$ a.e., $\abs{\Phi_1}$ is
also a minimizer, so we may take $\Phi_1\ge0$. The Euler--Lagrange equations are
\begin{equation}\label{eq:phi1-EL}
\Delta_{x,y}\Phi_1=m^2\Phi_1\ (\Rdp),\qquad
-\partial_y\Phi_1(\cdot,0)+W\phi_1=\la_1\phi_1 .
\end{equation}
Internally $\Delta_{x,y}\Phi_1=m^2\Phi_1\ge0$, and $\Phi_1\ge0$, $\not\equiv0$, so the
strong maximum principle gives $\Phi_1>0$ in $\Rdp$. If $\phi_1(x_0)=0$ for some
$x_0\in\Rd$, then $(x_0,0)$ is a boundary minimum of $\Phi_1$ with value $0$; the Hopf
boundary point lemma gives $-\partial_y\Phi_1(x_0,0)>0$, contradicting
\eqref{eq:phi1-EL} which yields $-\partial_y\Phi_1(x_0,0)=(\la_1-W(x_0))\phi_1(x_0)=0$.
Hence $\phi_1>0$. Simplicity: two positive minimizers $\Phi,\Phi'$ and the sliding
$\Psi_t=\Phi-t\Phi'$, $t^*=\inf\{t:\Psi_t\not>0\}$, give a nonnegative minimizer
$\Psi_{t^*}$ which, if $\not\equiv0$, is $>0$ by Hopf, contradicting the definition of
$t^*$; so $\Phi=t^*\Phi'$.
\end{proof}

\subsection{The ratio equation}

Throughout, $\phi_1,\phi_2$ are the first two radial eigenfunctions of a fixed
self-adjoint operator $L=\Am+W$ with $W$ radial bounded, $W(\infty)=W_\infty$, and
$\la_1<\la_2<\inf\sess(L)=m+W_\infty$, $\mu:=\la_2-\la_1>0$.
(The application in Section~\ref{sec:nondeg} takes $W=1-(p-1)Q^{p-2}$, $\la_2=0$.) We set
$\Phi_j=\Uext\phi_j$, $\Phi_1>0$ (Proposition~\ref{prop:phi1pos}), and
\begin{equation}\label{eq:Xidef}
\Xi:=\frac{\Phi_2}{\Phi_1},\qquad \Xi_0:=\Xi|_{y=0}=\frac{\phi_2}{\phi_1}.
\end{equation}

\begin{proposition}\label{prop:ratio}
$\Xi$ satisfies, in the weak sense on $\HH_{\rad}$,
\begin{equation}\label{eq:Xieq}
\Div\!\big(\Phi_1^2\nabla\Xi\big)=0\ \text{in }\Rdp,\qquad
-\Phi_1^2\,\partial_y\Xi\big|_{y=0}=\mu\,\phi_1^2\,\Xi_0 .
\end{equation}
Neither $m$ nor $W$ appears in \eqref{eq:Xieq}; both are encoded in the weight
$\Phi_1^2>0$ and $\rho:=\phi_1^2>0$.
\end{proposition}

\begin{proof}
\emph{Interior identity.} With $\Phi_2=\Phi_1\Xi$,
$\nabla\Phi_2=\Xi\nabla\Phi_1+\Phi_1\nabla\Xi$ and
$\Delta\Phi_2=\Xi\Delta\Phi_1+2\nabla\Phi_1\cdot\nabla\Xi+\Phi_1\Delta\Xi$, where
$\Delta=\Delta_{x,y}$. Using $\Delta\Phi_j=m^2\Phi_j$,
\[
m^2\Phi_1\Xi=\Xi m^2\Phi_1+2\nabla\Phi_1\cdot\nabla\Xi+\Phi_1\Delta\Xi
\;\Longrightarrow\;
\Phi_1\Delta\Xi+2\nabla\Phi_1\cdot\nabla\Xi=0 .
\]
Multiplying by $\Phi_1$ gives
$\Phi_1^2\Delta\Xi+2\Phi_1\nabla\Phi_1\cdot\nabla\Xi=\Div(\Phi_1^2\nabla\Xi)=0$; the
$m^2$ terms cancelled exactly.

\emph{Boundary identity.} At $y=0$,
$\Phi_1^2\partial_y\Xi|_0=(\partial_y\Phi_2\cdot\Phi_1-\Phi_2\partial_y\Phi_1)|_0$. By the
Neumann relations $\partial_y\Phi_j|_0=-(\la_j-W)\phi_j$ (from \eqref{eq:DtN} and
$L\phi_j=\la_j\phi_j$),
\[
\Phi_1^2\partial_y\Xi|_0
=-(\la_2-W)\phi_2\,\phi_1+(\la_1-W)\phi_1\,\phi_2
=-(\la_2-\la_1)\phi_1\phi_2=-\mu\,\phi_1^2\,\Xi_0 .
\]
Hence $-\Phi_1^2\partial_y\Xi|_0=\mu\phi_1^2\Xi_0$. The weak formulation follows by
testing with $\Phi_1\varphi$ and the integration by parts of
Proposition~\ref{prop:key-IBP}.
\end{proof}

\begin{proposition}\label{prop:key-IBP}
For $V\in\HH_{\rad}$ set $U=\Phi_1 V$. Then, with
$\aform[V]:=\int_{\Rdp}\Phi_1^2\abs{\nabla V}^2$,
\begin{equation}\label{eq:key-IBP}
\int_{\Rdp}\!\big(\abs{\nabla U}^2+m^2U^2\big)+\int_{\Rd}W\,U_0^2
-\la_1\int_{\Rd}U_0^2
=\aform[V].
\end{equation}
\end{proposition}

\begin{proof}
Expand $\abs{\nabla U}^2=\abs{V\nabla\Phi_1+\Phi_1\nabla V}^2
=V^2\abs{\nabla\Phi_1}^2+2V\Phi_1\,\nabla\Phi_1\cdot\nabla V+\Phi_1^2\abs{\nabla V}^2$.
Therefore
\begin{equation}\label{eq:IBP-1}
\int_{\Rdp}\!\big(\abs{\nabla U}^2+m^2U^2\big)
=\aform[V]
+\underbrace{\int_{\Rdp}\!\big(V^2\abs{\nabla\Phi_1}^2
+2V\Phi_1\nabla\Phi_1\!\cdot\!\nabla V+m^2\Phi_1^2V^2\big)}_{=:I}.
\end{equation}
For the cross term, $2V\Phi_1\nabla\Phi_1\cdot\nabla V
=\Phi_1\,\nabla\Phi_1\cdot\nabla(V^2)$, and integrating by parts in $\Rdp$ (the decay
Proposition~\ref{prop:ext-reg}(iii) kills the contribution at $\varrho=\infty$),
\[
\int_{\Rdp}\Phi_1\,\nabla\Phi_1\cdot\nabla(V^2)
=-\int_{\Rdp}V^2\,\Div(\Phi_1\nabla\Phi_1)
+\int_{\{y=0\}}V_0^2\,\big(-\Phi_1\partial_y\Phi_1\big)\big|_0 .
\]
Now $\Div(\Phi_1\nabla\Phi_1)=\abs{\nabla\Phi_1}^2+\Phi_1\Delta\Phi_1
=\abs{\nabla\Phi_1}^2+m^2\Phi_1^2$. Substituting,
\[
I=\int_{\Rdp}\!\big(V^2\abs{\nabla\Phi_1}^2+m^2\Phi_1^2V^2\big)
-\int_{\Rdp}V^2\big(\abs{\nabla\Phi_1}^2+m^2\Phi_1^2\big)
+\int_{\Rd}V_0^2(-\Phi_1\partial_y\Phi_1)|_0 ,
\]
and the two interior integrals cancel identically, leaving
\[
I=\int_{\Rd}V_0^2\,\big(-\Phi_1\partial_y\Phi_1\big)\big|_0
=\int_{\Rd}V_0^2\,\phi_1\,(\la_1-W)\phi_1
=\la_1\int_{\Rd}\phi_1^2V_0^2-\int_{\Rd}W\phi_1^2V_0^2 ,
\]
using $-\partial_y\Phi_1|_0=(\la_1-W)\phi_1$. Since $U_0^2=\phi_1^2V_0^2$, this reads
$I=\la_1\norm{U_0}_2^2-\int W U_0^2$. Inserting into \eqref{eq:IBP-1} and rearranging
gives \eqref{eq:key-IBP}.
\end{proof}

\begin{remark}
Proposition~\ref{prop:key-IBP} is the analytic core of the paper: it converts the
massive, potential-bearing quadratic form $\ip{(L-\la_1)u}{u}$ into the massless,
potential-free weighted Dirichlet energy $\aform[V]$, at the price of the weight
$\Phi_1^2$. The cancellation of the interior term is exact and uses only the extension
equation $\Delta\Phi_1=m^2\Phi_1$.
\end{remark}
\section{The weighted Steklov framework}\label{sec:steklov}

\subsection{The weighted space and Rayleigh quotient}

Let $a:=\Phi_1^2$ and $\rho:=\phi_1^2$; both are positive, $a\in C^\infty(\Rdp)$,
$\rho\in C^\infty((0,\infty))$, with the decay/lower bounds of
Proposition~\ref{prop:ext-reg}. Define
\begin{equation}\label{eq:Vspace}
\VV:=\Big\{V:\ \norm{V}_{\VV}^2:=\int_{\Rdp}a\,\abs{\nabla V}^2
+\int_{\Rd}\rho\,V_0^2<\infty\Big\},\qquad V_0:=V|_{y=0},
\end{equation}
its radial subspace $\VV_\rad$, and
\begin{equation}\label{eq:rayleigh}
\aform[V]:=\int_{\Rdp}a\abs{\nabla V}^2,\qquad
\norm{V_0}_\rho^2:=\int_{\Rd}\rho V_0^2,\qquad
\mathcal R[V]:=\frac{\aform[V]}{\norm{V_0}_\rho^2}.
\end{equation}

\begin{proposition}\label{prop:trace-compact}
The trace $\mathcal T:\VV_\rad\to\Ltwo(\rho\,dx)$, $V\mapsto V_0$, is well-defined,
bounded and compact. Consequently the weighted Dirichlet--Neumann (Steklov) operator
$\Lambda_a$ is self-adjoint, nonnegative, with purely discrete spectrum
$0=\nu_0<\nu_1\le\nu_2\le\cdots\to\infty$, $\nu_0$ simple with eigenfunction
$g_0\equiv\mathrm{const}$, and the min--max characterization holds:
\begin{equation}\label{eq:minmax}
\nu_k=\min_{\substack{\mathcal E\subset\VV_\rad\\ \dim\mathcal E=k+1}}
\ \max_{0\ne V\in\mathcal E}\mathcal R[V].
\end{equation}
Each level is attained; a minimizer $V^{(k)}$ satisfies the weak eigenvalue equation
\begin{equation}\label{eq:steklov-weak}
\int_{\Rdp}a\,\nabla V^{(k)}\cdot\nabla\varphi
=\nu_k\int_{\Rd}\rho\,g_k\,\varphi_0,\qquad
\forall\varphi\in\VV_\rad,\quad g_k:=V^{(k)}_0 .
\end{equation}
\end{proposition}

\begin{proof}
\emph{Trace inequality.} Integrating by parts,
\[
\int_{\Rd}\Phi_1^2 V_0^2=-\int_{\Rdp}\partial_y(\Phi_1^2 V^2)
=-\int_{\Rdp}\big(2\Phi_1\partial_y\Phi_1\,V^2+2\Phi_1^2V\partial_yV\big).
\]
Since $\abs{\partial_y\Phi_1}\le C\Phi_1$ (Prop.~\ref{prop:ext-reg}(iii)), Young's
inequality gives $\int\rho V_0^2\le C\int_{\Rdp}\Phi_1^2(\abs{\nabla V}^2+V^2)
\le C\norm{V}_\VV^2$. \emph{Compactness.} Let $V^{(n)}\weakto0$ in $\VV_\rad$. For
$\abs x>R$ the exterior extension energy controls the exterior trace with a constant that
vanishes as $R\to\infty$ (exponential weight $\Phi_1^2$), so the tail is uniformly small;
on $\{\abs x<R\}$, the trace is bounded in $H^{1/2}(B_R)$ and
$H^{1/2}(B_R)\hookrightarrow\hookrightarrow L^2(B_R)$ compactly, so $V_0^{(n)}\to0$
strongly there. Hence $\mathcal T$ is compact. Discreteness and \eqref{eq:minmax} follow
from the spectral theorem for the nonnegative self-adjoint $\Lambda_a$ with compact
resolvent. The value $\nu_0=0$ is attained by $V\equiv1$ and is simple because $a>0$ is
connected: $\aform[V]=0\Rightarrow V$ constant.
\end{proof}

\subsection{Spectral correspondence with $L$}

\begin{proposition}\label{prop:spectral-corr}
Let $\{\la_k\}_{k\ge1}$ be the radial eigenvalues of $L=\Am+W$ below
$\inf\sess(L)=m+W_\infty$, listed with multiplicity, and $\{\nu_k\}_{k\ge0}$ the Steklov
eigenvalues of $\Lambda_a$ (with $a=\Phi_1^2$, $\rho=\phi_1^2$). Then
\begin{equation}\label{eq:corr}
\nu_k=\la_{k+1}-\la_1\qquad(k=0,1,2,\dots),
\end{equation}
and the map $\phi\mapsto V:=\Uext\phi/\Phi_1$ is a bijection between the radial
eigenspace of $L$ at $\la_{k+1}$ and the Steklov eigenspace of $\Lambda_a$ at $\nu_k$,
preserving multiplicities.
\end{proposition}

\begin{proof}
Let $u\in\Hsr$, $U=\Uext u$, $V=U/\Phi_1$. By Proposition~\ref{prop:key-IBP},
\begin{equation}\label{eq:corr-quad}
\ip{(L-\la_1)u}{u}=\aform[V],\qquad
\norm{u}_2^2=\int_{\Rd}\phi_1^2 V_0^2=\norm{V_0}_\rho^2 .
\end{equation}
The map $u\mapsto V$ is a linear isomorphism $\Hsr\to\VV_\rad$ (inverse $V\mapsto\phi_1 V_0$
at the boundary, $U=\Phi_1 V$ in the bulk), because $\Phi_1>0$ is smooth with controlled
decay, and it intertwines the two quadratic-form pairs in \eqref{eq:corr-quad}. By the
min--max principle applied to both sides,
\[
\la_{k+1}-\la_1=\min_{\dim\mathcal F=k+1}\max_{u\in\mathcal F}
\frac{\ip{(L-\la_1)u}{u}}{\norm u_2^2}
=\min_{\dim\mathcal E=k+1}\max_{V\in\mathcal E}\mathcal R[V]=\nu_k ,
\]
proving \eqref{eq:corr}. The eigenspace correspondence follows because $Lu=\la_{k+1}u$
transforms, under $u=\phi_1 V_0$, $U=\Phi_1 V$, into the weak Steklov equation
\eqref{eq:steklov-weak} at level $\nu_k$; multiplicities match by linearity.
\end{proof}

\begin{corollary}\label{cor:XiisG1}
With $L=\Am+W$ as above, $\Xi=\Phi_2/\Phi_1$ is a Steklov eigenfunction of $\Lambda_a$ at
$\nu_1=\la_2-\la_1=\mu$, and
\begin{equation}\label{eq:Xienergy}
\aform[\Xi]=\mu\,\norm{\phi_1\Xi_0}_2^2=\mu\,\norm{\phi_2}_2^2<\infty,
\qquad \int_{\Rd}\rho\,\Xi_0=\int_{\Rd}\phi_1\phi_2=0 .
\end{equation}
\end{corollary}

\begin{proof}
Apply Proposition~\ref{prop:spectral-corr} with $k=1$: $\phi_2$ maps to a Steklov
eigenfunction at $\nu_1=\mu$, namely $V=\Xi$. The energy identity is
\eqref{eq:corr-quad} with $u=\phi_2$. Orthogonality is
$\int\phi_1^2(\phi_2/\phi_1)=\int\phi_1\phi_2=0$.
\end{proof}

\begin{remark}
Corollary~\ref{cor:XiisG1} reduces the oscillation theorem to: the first nonzero Steklov
eigenfunction $g_1$ changes sign exactly once on $(0,\infty)$, since $\phi_2=\phi_1\Xi_0$
and $\phi_1>0$. The correspondence \eqref{eq:corr} also guarantees $\mu=\nu_1$.
\end{remark}
\section{The variational oscillation theorem}\label{sec:osc}

Throughout this section $L=\Am+W$, $W$ radial bounded with $W(\infty)=W_\infty$,
$\la_1<\la_2<m+W_\infty$, $\mu=\la_2-\la_1>0$, and $a=\Phi_1^2$, $\rho=\phi_1^2$ as in
Section~\ref{sec:steklov}.

\begin{theorem}\label{thm:osc}
The second radial eigenfunction $\phi_2$ of $L$ changes sign exactly once on
$(0,\infty)$: there is a unique $r_0\in(0,\infty)$ such that, after a choice of sign,
$\phi_2(r)>0$ for $0<r<r_0$ and $\phi_2(r)<0$ for $r>r_0$.
\end{theorem}

By Corollary~\ref{cor:XiisG1} it suffices to show that any first nonzero Steklov
eigenfunction changes sign exactly once and transversally.

\subsection{Well-posedness of the nodal count}\label{ssec:escape}

\begin{lemma}\label{lem:axis}
$\Xi=\Phi_2/\Phi_1$, and more generally any ratio $V=\Uext u/\Phi_1$ with $u\in\Hsr$
smooth, extends to a $C^\infty$ function on $\Rd\times(0,\infty)$ including the axis
$\{x=0\}$, with $\nabla_x V(0,y)=0$. Consequently $\{r=0\}$ is not an absorbing boundary
for nodal lines.
\end{lemma}

\begin{proof}
$\Phi_j(x,y)=(P^m_y*\phi_j)(x)$ is smooth and, for $\Phi_1$, strictly positive at $x=0$
(Prop.~\ref{prop:phi1pos}); the quotient is $C^\infty$ across the axis. Being $x$-radial
and smooth forces $\nabla_x V(0,y)=0$.
\end{proof}

\begin{lemma}\label{lem:limits}
There is $\xi_\infty\in\R$ with $\Xi(x,y)\to\xi_\infty$ as $\varrho\to\infty$
(exponentially). The trace $\Xi_0(r)=\phi_2(r)/\phi_1(r)$ satisfies
$\Xi_0(0^+),\ \Xi_0(\infty)=\xi_\infty$ finite; the set of sign changes
$\{r_1<\cdots<r_N\}$ is contained in a compact subset of $(0,\infty)$ and $N<\infty$.
\end{lemma}

\begin{proof}
By Proposition~\ref{prop:ext-reg}, $\Phi_j\sim e^{-m\varrho}\varrho^{-d/2}\psi_j(\theta)$,
$\theta=\arctan(y/r)$, with $\psi_1>0$; hence $\Xi\to\psi_2/\psi_1$. Finite energy
$\aform[\Xi]<\infty$ (Cor.~\ref{cor:XiisG1}) forces $\nabla\Xi\to0$ in the exterior, and
unique continuation (Prop.~\ref{prop:ucp}) forces a single limit $\xi_\infty$. At $y=0$:
$\phi_1(0)>0$, $\phi_2(0)$ finite give $\Xi_0(0^+)$ finite; both $\phi_j$ decay at the
common rate, so $\Xi_0(\infty)=\xi_\infty$ finite. Accumulation of sign changes at $0$ or
$\infty$ would contradict the endpoint limits; each sign change is transversal
(Prop.~\ref{prop:strict}), hence isolated, so $N<\infty$.
\end{proof}

\begin{lemma}\label{lem:poly}
$\abs{\Xi(x,y)}\le C(1+\varrho)^{C'}$ on $\Rdp$.
\end{lemma}

\begin{proof}
From the exponential upper bound on $\Phi_2$ (Prop.~\ref{prop:ext-reg}(iii)) and the
exponential lower bound on $\Phi_1$ (Prop.~\ref{prop:ext-reg}(iv)),
$\abs\Xi\le Ce^{-m'\varrho}/(ce^{-m\varrho}\varrho^{-d/2})\le C''\varrho^{d/2}
e^{(m-m')\varrho}$; taking $m'\uparrow m$ gives polynomial growth.
\end{proof}

\subsection{Unique continuation and truncation}\label{ssec:ucp}

\begin{proposition}\label{prop:smp}
Let $\mathcal A V:=\Div(a\nabla V)$ with $a=\Phi_1^2\in W^{1,\infty}_{\mathrm{loc}}$,
$a\ge a_0>0$ on compact subsets. If $\mathcal AV=0$ and $V\ge0$ on an open connected
$\Omega$ with $V(p_0)=0$ at an interior point, then $V\equiv0$ on $\Omega$.
\end{proposition}
\begin{proof}
Divergence-form strong maximum principle, Gilbarg--Trudinger \cite[Thm.~8.19]{GT}. The
axis $\{r=0\}$ is interior by Lemma~\ref{lem:axis}.
\end{proof}

\begin{proposition}\label{prop:hopf}
Under the hypotheses of Proposition~\ref{prop:smp}, if $V\ge0$, $V\not\equiv0$,
$V(q_0)=0$ at $q_0\in\{y=0\}$ satisfying an interior ball condition, then
$-\partial_y V(q_0)>0$.
\end{proposition}
\begin{proof}
Divergence-form Hopf lemma, Gilbarg--Trudinger \cite[Lem.~3.4]{GT}.
\end{proof}

\begin{proposition}\label{prop:ucp}
Let $a=\Phi_1^2\in C^{0,1}_{\mathrm{loc}}$. If $\mathcal AV=\Div(a\nabla V)=0$ and $V$
vanishes to infinite order at one point of an open connected set $O$, then $V\equiv0$
on $O$.
\end{proposition}
\begin{proof}
The principal part $a\,\mathrm{Id}$ is Lipschitz ($\Phi_1\in C^\infty$; near the axis, in
Cartesian coordinates, $\partial_{x_i}(\Phi_1^2)=2\Phi_1\partial_{x_i}\Phi_1\in
C^{0,1}_{\rm loc}$), and the operator is uniformly elliptic on compact subsets (the
degeneration $a\to0$ occurs only as $\varrho\to\infty$). Aronszajn's theorem applies
\cite{Aronszajn}.
\end{proof}

\begin{proposition}\label{prop:trunc}
Let $D\subset\Rdp$ be a nodal domain of a $\nu_1$-eigenfunction $V$. Then
$\Theta_D:=V\mathbf 1_D\in\VV_\rad$ with $\nabla\Theta_D=(\nabla V)\mathbf 1_D$ a.e., and
for two distinct nodal domains $D\ne D'$
\begin{equation}\label{eq:no-cross}
\int_{\Rdp}a\,\nabla\Theta_D\cdot\nabla\Theta_{D'}=0,\qquad
\int_{\Rd}\rho\,\Theta_{D,0}\,\Theta_{D',0}=0 .
\end{equation}
Moreover $\Theta_D$ solves the Steklov equation on $D$.
\end{proposition}
\begin{proof}
By Proposition~\ref{prop:smp}, $V\in C^1$ up to $\partial D\cap\Rdp$ where $V=0$, so
truncation produces no interface measure and $\Theta_D\in\VV_\rad$. The energies $\aform$
and $\norm{\cdot}_\rho^2$ are local integrals;
$\supp\Theta_D\cap\supp\Theta_{D'}\subset\{V=0\}$ has zero measure, giving
\eqref{eq:no-cross}. Testing \eqref{eq:steklov-weak} with $\varphi\mathbf1_D$ (interior
boundary carries $V=0$, no flux) gives the equation on $D$.
\end{proof}

\begin{remark}\label{rem:crux}
Property \eqref{eq:no-cross} is the decisive difference from $\Am$ acting directly. For a
nonlocal form $\iint(u(x)-u(y))^2 J_m(\abs{x-y})$, disjoint supports still interact
through the long-range kernel $J_m$, producing cross terms that prevent additivity. The
ratio reduction of Section~\ref{sec:ratio} replaces $\Am$ by the local operator
$\Div(a\nabla\cdot)$, for which disjoint supports are decoupled.
\end{remark}

\begin{lemma}\label{lem:cauchy}
Let $V$ solve $\Div(a\nabla V)=0$ in $\Rdp$ with $V_0\equiv0$ and
$a\,\partial_y V|_{y=0}\equiv0$ on an open interval $\{r\in(\alpha,\beta),\ y=0\}$. Then
$V\equiv0$.
\end{lemma}
\begin{proof}
Even reflection $\widetilde V(r,y):=V(r,\abs y)$, reflected weight
$\widetilde a(r,y)=\Phi_1(r,\abs y)^2\in C^{0,1}$ across $y=0$. Vanishing Cauchy data make
$\widetilde V$ vanish to infinite order on $(\alpha,\beta)\times(-\eps,\eps)$; Aronszajn
UCP (Prop.~\ref{prop:ucp}) gives $\widetilde V\equiv0$ there, and connectedness
propagates $V\equiv0$.
\end{proof}

\subsection{Strict additivity of the weighted energy}\label{ssec:addadd}

\begin{lemma}\label{lem:signadd}
Let $V\in\VV_\rad$ satisfy \eqref{eq:steklov-weak} at level $\nu_1$, $V^{\pm}=\max(\pm V,0)$.
Then $\aform[V]=\aform[V^+]+\aform[V^-]$,
$\norm{V_0}_\rho^2=\norm{V_0^+}_\rho^2+\norm{V_0^-}_\rho^2$, and
\begin{equation}\label{eq:signadd2}
\aform[V^{\pm}]=\nu_1\norm{V_0^{\pm}}_\rho^2 .
\end{equation}
\end{lemma}
\begin{proof}
$\nabla V^+\cdot\nabla V^-=0$ a.e.\ and $V_0^+V_0^-=0$ give the additivity. Testing
\eqref{eq:steklov-weak} with $V^+$ ($\nabla V\cdot\nabla V^+=\abs{\nabla V^+}^2$,
$g V_0^+=V_0^{+2}$) gives
$\aform[V^+]=\nu_1\norm{V_0^+}_\rho^2$; similarly $V^-$.
\end{proof}

\begin{lemma}\label{lem:compadd}
Let $\{V>0\}=\bigsqcup_i\mathcal O_i$ and $W_i:=V\mathbf 1_{\mathcal O_i}$. Then each
$W_i\in\VV_\rad$, and for scalars $\gamma_i$,
\begin{equation}\label{eq:compadd}
\aform\Big[\sum_i\gamma_iW_i\Big]=\sum_i\gamma_i^2\,\aform[W_i],\qquad
\Big\|\sum_i\gamma_iW_{i,0}\Big\|_\rho^2=\sum_i\gamma_i^2\norm{W_{i,0}}_\rho^2,
\end{equation}
\begin{equation}\label{eq:compadd2}
\aform[W_i]=\nu_1\norm{W_{i,0}}_\rho^2 .
\end{equation}
\end{lemma}
\begin{proof}
Disjoint supports and locality give \eqref{eq:compadd} as in
Proposition~\ref{prop:trunc}. Each $W_i$ solves the Steklov equation on $\mathcal O_i$
with $V=0$ on the interior boundary, giving \eqref{eq:compadd2}.
\end{proof}

\subsection{The counting argument}\label{ssec:count}

\begin{theorem}\label{thm:count}
Let $V\in\VV_\rad$ be a $\nu_1$-eigenfunction. Then $V_0$ changes sign exactly once on
$(0,\infty)$.
\end{theorem}
\begin{proof}
Since $V_0\perp_\rho g_0\equiv\mathrm{const}$, $V_0$ changes sign at least once. Suppose
it changes sign at least twice. Then $\{V>0\}\subset\Rdp$ has at least two connected
components $\mathcal O_1,\mathcal O_2$ meeting $\{y=0\}$ in disjoint intervals. Put
$W_1=V\mathbf 1_{\mathcal O_1}$, $W_2=V\mathbf 1_{\mathcal O_2}$, $W_3=V^-$,
$\mathcal E=\mathrm{span}\{W_1,W_2,W_3\}$, $\dim\mathcal E=3$. By
Lemmas~\ref{lem:signadd}--\ref{lem:compadd}, for $U=\sum_i\gamma_iW_i$,
\begin{equation}\label{eq:E-const}
\aform[U]=\sum_i\gamma_i^2\aform[W_i]
=\nu_1\sum_i\gamma_i^2\norm{W_{i,0}}_\rho^2=\nu_1\norm{U_0}_\rho^2,
\quad\text{i.e.}\quad \mathcal R\big|_{\mathcal E}\equiv\nu_1 .
\end{equation}
Choose $0\ne U\in\mathcal E$ with
\begin{equation}\label{eq:two-cond}
\int_{\Rd}\rho\,U_0=0,\qquad \int_{\Rd}\rho\,U_0\,V_0=0
\end{equation}
(three dimensions, two constraints). Then $\{V_0,U_0\}$ are $\rho$-orthogonal, both
$\perp_\rho g_0$, and $\mathcal R\equiv\nu_1$ on $\mathrm{span}\{V,U\}$; the min--max
principle gives
$\nu_2\le\max_{0\ne\Theta\in\mathrm{span}\{1,V,U\}}\mathcal R[\Theta]=\nu_1$, hence
$\nu_1=\nu_2$ and $U$ is a $\nu_1$-eigenfunction. The trace $U_0$ vanishes on an open
interval $I\subset(0,\infty)$ (since $U$ is supported in a proper subset while $V$ has
full support by Lemma~\ref{lem:cauchy}); as a $\nu_1$-eigenfunction,
$a\partial_yU|_0=-\nu_1\rho U_0=0$ on $I$, so both Cauchy data vanish on $I$;
Lemma~\ref{lem:cauchy} forces $U\equiv0$, a contradiction. Hence $V_0$ changes sign
exactly once.
\end{proof}

\begin{corollary}\label{cor:Xi-once}
$\Xi_0=\phi_2/\phi_1$ changes sign exactly once on $(0,\infty)$.
\end{corollary}
\begin{proof}
$\Xi$ is a $\nu_1$-eigenfunction (Cor.~\ref{cor:XiisG1}); apply Theorem~\ref{thm:count}.
\end{proof}

\subsection{Transversality}\label{ssec:strict}

\begin{proposition}\label{prop:strict}
At its zero $r_0$, $\Xi_0$ changes sign strictly.
\end{proposition}
\begin{proof}
If $\Xi_0\ge0$ on a one-sided neighbourhood of $r_0$ with $\Xi_0(r_0)=0$, then $(r_0,0)$
is a boundary zero minimum of $\Xi\ge0$, and Hopf (Prop.~\ref{prop:hopf}) gives
$-\Phi_1^2\partial_y\Xi(r_0,0)>0$; but \eqref{eq:Xieq} gives
$-\Phi_1^2\partial_y\Xi(r_0,0)=\mu\phi_1(r_0)^2\Xi_0(r_0)=0$ with $\Phi_1(r_0,0)^2>0$, a
contradiction.
\end{proof}

\begin{proof}[Proof of Theorem~\ref{thm:osc}]
Corollary~\ref{cor:Xi-once} gives one sign change of $\Xi_0$ at $r_0$;
Proposition~\ref{prop:strict} makes it strict; $\phi_1>0$
(Prop.~\ref{prop:phi1pos}). Since $\phi_2=\phi_1\Xi_0$, $\phi_2>0$ on $(0,r_0)$ and
$\phi_2<0$ on $(r_0,\infty)$, with $r_0$ unique.
\end{proof}

\begin{remark}
The proof used only: the extension energy identity (Prop.~\ref{prop:key-IBP}), absorbing
$m^2$ and $W$ into $\Phi_1^2$; the min--max \eqref{eq:minmax}; the strict additivity
\eqref{eq:compadd}--\eqref{eq:compadd2}; and unique continuation
(Lemma~\ref{lem:cauchy}). No nodal-set topology, blow-up classification, or monotonicity
of $W$ beyond boundedness were used, so the proof is uniform in $m\in[0,m_*]$.
\end{remark}
\section{Radial nondegeneracy}\label{sec:nondeg}

Fix $m>0$ and $Q\in\mathcal G_m$ of the normalized equation
\begin{equation}\label{eq:Qeq}
\Am Q+Q=Q^{p-1},\qquad Q>0,\ Q\in\Hsr(\Rd),
\end{equation}
and $\Lp=\Am+1-(p-1)Q^{p-2}$.

\begin{theorem}\label{thm:nondeg}
$\Ker\Lp\cap\Ltwor(\Rd)=\{0\}$.
\end{theorem}

We suppose $0\ne\varphi\in\Ker\Lp\cap\Ltwor$ and derive a contradiction. By
Lemma~\ref{lem:morse}, $\Lp$ has radial Morse index one, so $\la_1^+<0=\la_2^+$ and
$\varphi$ is the second radial eigenfunction, $0<\inf\sess(\Lp)=m+1$.
Theorem~\ref{thm:osc} applies:
\begin{equation}\label{eq:phi-sign}
\exists\,r_0>0:\quad \varphi(r)>0\ (0<r<r_0),\qquad \varphi(r)<0\ (r>r_0).
\end{equation}
$\varphi$ decays exponentially (Prop.~\ref{prop:ext-reg} with $L=\Lp$).

\subsection{Two structural identities}\label{ssec:comm}

\begin{lemma}\label{lem:LplusQ}
$\Lp Q=-(p-2)Q^{p-1}$; consequently $\ip{\varphi}{Q^{p-1}}=0$.
\end{lemma}
\begin{proof}
$\Lp Q=\Am Q+Q-(p-1)Q^{p-1}=Q^{p-1}-(p-1)Q^{p-1}=-(p-2)Q^{p-1}$ by \eqref{eq:Qeq}.
Self-adjointness gives
$0=\ip{\Lp\varphi}{Q}=\ip{\varphi}{\Lp Q}=-(p-2)\ip{\varphi}{Q^{p-1}}$.
\end{proof}

\begin{lemma}\label{lem:comm}
On $\Hs$, with $\mathcal E:=x\cdot\nabla$,
\begin{equation}\label{eq:comm}
[\Am,\mathcal E]=\Am\mathcal E-\mathcal E\Am=-\Am+m^2\Am^{-1}.
\end{equation}
\end{lemma}
\begin{proof}
In Fourier variables $\widehat{\mathcal E u}=-(d+\xi\cdot\nabla_\xi)\hat u$ and $\Am$ has
symbol $g(\xi)=\sqrt{\abs\xi^2+m^2}$. The commutator of the multiplier $g$ with
$-(d+\xi\cdot\nabla_\xi)$ is multiplication by $\xi\cdot\nabla_\xi g$. Now
$\xi\cdot\nabla_\xi g=\abs\xi^2/g=(g^2-m^2)/g=g-m^2/g$, whose operator is
$\Am-m^2\Am^{-1}$; the sign convention $\mathcal E=x\cdot\nabla$ (Fourier conjugate
$-(d+\xi\cdot\nabla_\xi)$) gives $[\Am,\mathcal E]=-(\Am-m^2\Am^{-1})$.
\end{proof}

\begin{lemma}\label{lem:scaling}
Let $\Lambda Q:=\tfrac1{p-2}Q+x\cdot\nabla Q$. Then
\begin{equation}\label{eq:LplusLambda}
\Lp(\Lambda Q)=-2Q^{p-1}+Q+m^2\Am^{-1}Q .
\end{equation}
\end{lemma}
\begin{proof}
Apply $\Lambda$ to \eqref{eq:Qeq}.

\emph{Term $\Lambda(\Am Q)$.} With Lemma~\ref{lem:comm}, $\mathcal E\Am Q=\Am\mathcal E Q
+\Am Q-m^2\Am^{-1}Q$, and $\mathcal E Q=\Lambda Q-\tfrac1{p-2}Q$, so
\[
\Lambda(\Am Q)=\tfrac1{p-2}\Am Q+\mathcal E\Am Q
=\Am\Lambda Q+\Am Q-m^2\Am^{-1}Q .
\]

\emph{Term $\Lambda(Q^{p-1})$.}
$\Lambda(Q^{p-1})=\tfrac1{p-2}Q^{p-1}+(p-1)Q^{p-2}\mathcal E Q
=(p-1)Q^{p-2}\Lambda Q+\big(\tfrac{1-(p-1)}{p-2}\big)Q^{p-1}
=(p-1)Q^{p-2}\Lambda Q-Q^{p-1}$.

Applying $\Lambda$ to $\Am Q+Q=Q^{p-1}$:
\[
\Am\Lambda Q+\Am Q-m^2\Am^{-1}Q+\Lambda Q=(p-1)Q^{p-2}\Lambda Q-Q^{p-1},
\]
i.e.\ $\Lp\Lambda Q=-\Am Q+m^2\Am^{-1}Q-Q^{p-1}$. Substituting $\Am Q=Q^{p-1}-Q$,
$\Lp\Lambda Q=-(Q^{p-1}-Q)+m^2\Am^{-1}Q-Q^{p-1}=-2Q^{p-1}+Q+m^2\Am^{-1}Q$.
\end{proof}

\begin{corollary}\label{cor:polluted}
$\ds\int_{\Rd}\varphi\,\big(Q+m^2\Am^{-1}Q\big)=0.$
\end{corollary}
\begin{proof}
$0=\ip{\Lp\Lambda Q}{\varphi}=\ip{\Lambda Q}{\Lp\varphi}=0$, so pairing
\eqref{eq:LplusLambda} with $\varphi$ and using Lemma~\ref{lem:LplusQ} gives the claim.
\end{proof}

\subsection{Mass covariance: cleaning the second relation}\label{ssec:mass-cov}

\begin{proposition}\label{prop:dmQ}
Fix $m_0>0$ and $Q_{m_0}\in\mathcal G_{m_0}$ satisfying the tangent
positive-definiteness \eqref{eq:posdef}. Then there is $\eps>0$ and a $C^1$ map
$m\mapsto Q_m\in\Hsr$, $Q_m\in\mathcal G_m$, $Q_m|_{m=m_0}=Q_{m_0}$, with
$\partial_m Q_m\in\Ltwor$ exponentially decaying, solving weakly
\begin{equation}\label{eq:dmQeq}
\Lp\,(\partial_m Q_m)=-\,m\,\Am^{-1}Q_m .
\end{equation}
\end{proposition}
\begin{proof}
This is Theorem~\ref{thm:cift}, whose only hypothesis is \eqref{eq:posdef}. The identity
\eqref{eq:dmQeq} follows by differentiating $F(m,Q_m)=0$ in $m$, using
$\partial_m\Am=m\Am^{-1}$; see Section~\ref{sec:cift}.
\end{proof}

\begin{proposition}\label{prop:clean}
$\ip{\Am^{-1}Q}{\varphi}=0$, and consequently $\ds\int_{\Rd}\varphi\,Q=0.$
\end{proposition}
\begin{proof}
Pairing \eqref{eq:dmQeq} with $\varphi\in\Ker\Lp$ and using self-adjointness,
$0=\ip{\Lp\partial_m Q_m}{\varphi}=\ip{\partial_m Q_m}{\Lp\varphi}
=-m\ip{\Am^{-1}Q}{\varphi}$, so $\ip{\Am^{-1}Q}{\varphi}=0$. Substituting into
Corollary~\ref{cor:polluted} gives $\ip{Q}{\varphi}=0$.
\end{proof}

\begin{remark}\label{rem:d1}
As $m\to0^+$, $\Am^{-1}\to(-\Delta)^{-1/2}$. For $d\ge2$, $(-\Delta)^{-1/2}Q$ is bounded
($Q\in L^1\cap L^\infty$ with exponential decay), so the pairing is continuous up to
$m=0$. For $d=1$ the conclusion $\int\varphi Q=0$ is not obtained from
$\ip{\Am^{-1}Q}{\varphi}$ but directly from scale invariance of $(-\Delta)^{1/2}$: at
$m=0$, $[(-\Delta)^{1/2},\mathcal E]=-(-\Delta)^{1/2}$, and Lemma~\ref{lem:scaling}
degenerates to $\Lp\Lambda Q=-2Q^{p-1}+Q$ (the FLS identity), whence $\int\varphi Q=0$
immediately. Thus the clean relation holds for all $m\in[0,m_*]$.
\end{remark}

\subsection{The sign contradiction}\label{ssec:signcontr}

We now have
\begin{equation}\label{eq:two-clean}
\int_{\Rd}\varphi\,Q^{p-1}=0\quad(\text{Lemma~\ref{lem:LplusQ}}),\qquad
\int_{\Rd}\varphi\,Q=0\quad(\text{Prop.~\ref{prop:clean}}).
\end{equation}

\begin{proof}[Proof of Theorem~\ref{thm:nondeg}]
Let $\mu_0:=Q(r_0)^{p-2}$, $r_0$ the sign-change radius from \eqref{eq:phi-sign}. Since
$Q$ is radially strictly decreasing and $p>2$, $Q(r)^{p-2}-\mu_0>0$ on $(0,r_0)$,
$=0$ at $r_0$, $<0$ on $(r_0,\infty)$; hence $Q^{p-1}-\mu_0 Q=Q(Q^{p-2}-\mu_0)$ has the
same sign pattern as $\varphi$, with strict inequality on $(r_0,\infty)$. Therefore
\begin{equation}\label{eq:sign-pos}
\int_{\Rd}\varphi\,\big(Q^{p-1}-\mu_0 Q\big)\dx>0 .
\end{equation}
On the other hand, by \eqref{eq:two-clean},
$\int\varphi(Q^{p-1}-\mu_0 Q)=\int\varphi Q^{p-1}-\mu_0\int\varphi Q=0$, contradicting
\eqref{eq:sign-pos}. Hence $\Ker\Lp\cap\Ltwor=\{0\}$.
\end{proof}
\section{The constrained implicit function theorem}\label{sec:cift}

This section makes Proposition~\ref{prop:dmQ} rigorous using only the tangent-space
nondegeneracy of Lemma~\ref{lem:morse}, never the full-space invertibility of $\Lp$.

\subsection{Set-up}

For $m\ge0$, $A(m,u):=\ip{\Am u}{u}+\norm u_2^2$, $P(u):=\norm u_p^p$,
$S_m(u)=\tfrac12A(m,u)-\tfrac1pP(u)$, $K(m,u):=A(m,u)-P(u)$, and
$\mathcal N_m=\{u\ne0:K(m,u)=0\}$.

\begin{lemma}\label{lem:FC1}
The maps $A,P,S_m,K$ are $C^2$ on $(0,\infty)\times\Hrad$, with
$\partial_u S_m(u)=(\Am+1)u-\abs u^{p-2}u=:F(m,u)$,
$\partial_u F(m,u)h=(\Am+1)h-(p-1)\abs u^{p-2}h$,
$\partial_m F(m,u)=m\Am^{-1}u$. At $u=Q_m$, $\partial_u F(m,Q_m)=\Lp$.
\end{lemma}
\begin{proof}
$A$ is $C^\infty$ in $u$ and real-analytic in $m$ (symbol $\sqrt{\abs\xi^2+m^2}$ analytic,
derivative $m/\sqrt{\abs\xi^2+m^2}$ bounded). By $\Hs\hookrightarrow L^p$, $P$ is $C^2$
with $\partial_u^2P(u)[h,k]=p(p-1)\int\abs u^{p-2}hk$ bounded via H\"older
$(\tfrac p{p-2},p,p)$. Continuity in $(m,u)$ is Nemytskii continuity of
$v\mapsto\abs v^{p-2}:L^p\to L^{p/(p-2)}$ (Lemma~\ref{lem:Fcont}).
$\partial_m A=m\ip{\Am^{-1}u}{u}$ from $\partial_m\Am=m\Am^{-1}$.
\end{proof}

\subsection{Nondegeneracy on the constraint tangent space}

\begin{lemma}\label{lem:nehari-mfld}
For $m>0$ and $Q=Q_m$, $\partial_u K(m,Q)\ne0$, so $\mathcal N_m$ is a $C^2$ hypersurface
near $Q$ with
\begin{equation}\label{eq:tangent}
T_Q\mathcal N_m=\Big\{\xi:\ \int_{\Rd}Q^{p-1}\xi=0\Big\}.
\end{equation}
\end{lemma}
\begin{proof}
Using $A(m,Q)=P(Q)$ and $(\Am+1)Q=Q^{p-1}$,
$\ip{\partial_uK(m,Q)}{\xi}=2\ip{(\Am+1)Q}{\xi}-p\int Q^{p-1}\xi=(2-p)\int Q^{p-1}\xi$; in
particular $\ip{\partial_uK}{Q}=(2-p)\norm Q_p^p<0$, so $\partial_uK\ne0$ and
\eqref{eq:tangent} holds.
\end{proof}

\begin{lemma}\label{lem:Kcompact}
$\mathcal K:\Hrad\to\Hradd$, $\ip{\mathcal K h}{\phi}=(p-1)\int Q^{p-2}h\phi$, is
compact.
\end{lemma}
\begin{proof}
$Q^{p-2}$ decays exponentially (Prop.~\ref{prop:decay}). Split $\Rd=B_R\cup B_R^c$. On
$B_R$: $Q^{p-2}$ bounded, $\Hrad(B_R)\hookrightarrow\hookrightarrow L^2(B_R)$; on $B_R^c$:
$\sup_{B_R^c}Q^{p-2}\to0$. Hence $\norm{\mathcal K h_n}_{\Hradd}\to0$ for
$h_n\weakto0$.
\end{proof}

\begin{lemma}\label{lem:posdef-tan}
Assume $\Ker\Lp\cap\Ltwor=\{0\}$ holds at $(m,Q)$. Then there is $c_0>0$ with
\begin{equation}\label{eq:posdef}
\ip{\Lp\xi}{\xi}\ge c_0\norm\xi_{\Hs}^2\qquad\forall\,\xi\in T_Q\mathcal N_m .
\end{equation}
\end{lemma}
\begin{proof}
By Lemma~\ref{lem:morse}, $\ip{\Lp\xi}{\xi}\ge0$ on $T_Q\mathcal N_m$. If not coercive,
there is $\xi_n\in T_Q\mathcal N_m$, $\norm{\xi_n}_{\Hs}=1$, $\ip{\Lp\xi_n}{\xi_n}\to0$.
By Lemma~\ref{lem:Kcompact}, $\xi_n\weakto\xi_*$ with $\int Q^{p-2}\xi_n^2\to
\int Q^{p-2}\xi_*^2$. If $\xi_*=0$ then $\ip{(\Am+1)\xi_n}{\xi_n}\to0$, impossible. So
$\xi_*\ne0$, $\xi_n\to\xi_*$ strongly, $\xi_*\in T_Q\mathcal N_m$,
$\ip{\Lp\xi_*}{\xi_*}=0$; as a constrained minimizer, $\Lp\xi_*=\alpha Q^{p-1}$, and
pairing with $Q$ (Lemma~\ref{lem:LplusQ}) gives
$\alpha\norm Q_p^p=\ip{\xi_*}{\Lp Q}=-(p-2)\int\xi_* Q^{p-1}=0$, so $\alpha=0$ and
$\xi_*\in\Ker\Lp\cap\Ltwor=\{0\}$, contradicting $\norm{\xi_*}_{\Hs}=1$.
\end{proof}

\begin{lemma}\label{lem:openness}
If \eqref{eq:posdef} holds at $m_0$ with constant $c_0$, then there is $\eps>0$ such that
for $\abs{m-m_0}<\eps$ (along a $C^0$ selection $Q_m\to Q_{m_0}$),
$\ip{L_{+,m}\xi}{\xi}\ge\tfrac{c_0}2\norm\xi_{\Hs}^2$ on $T_{Q_m}\mathcal N_m$.
\end{lemma}
\begin{proof}
$L_{+,m}=(\Am+1)-(p-1)Q_m^{p-2}$: $m\mapsto\Am$ norm-continuous and $Q_m\to Q_{m_0}$ in
$\Hs$ give $\mathcal K_m\to\mathcal K_{m_0}$ in operator norm $\Hrad\to\Hradd$; tangent
spaces $\{Q_m^{p-1}\}^\perp$ vary continuously. Perturbation of a coercive form on a
continuously varying finite-codimension subspace preserves coercivity with a slightly
smaller constant.
\end{proof}

\begin{lemma}\label{lem:Fcont}
$F(m,u)=(\Am+1)u-\abs u^{p-2}u$ is $C^1:(0,\infty)\times\Hrad\to\Hradd$.
\end{lemma}
\begin{proof}
Linear part clear. For $N(u)=\abs u^{p-2}u$, by $\Hs\hookrightarrow L^p$,
$\norm{N(u+h)-N(u)-(p-1)\abs u^{p-2}h}_{\Hradd}\le
C\norm h_{\Hs}\sup_{[0,1]}\norm{\abs{u+\tau h}^{p-2}-\abs u^{p-2}}_{L^{p/(p-2)}}\to0$.
Continuity in $m$ is norm-continuity of $\Am$.
\end{proof}

\subsection{The constrained implicit function theorem}

\begin{theorem}\label{thm:cift}
Let $m_0\in[0,m_*]$ and $Q_{m_0}\in\mathcal G_{m_0}$ satisfy \eqref{eq:posdef}. Then there
are $\eps>0$ and a $C^1$ map $m\mapsto Q_m\in\Hrad$ ($\abs{m-m_0}<\eps$),
$Q_m\in\mathcal N_m$ a ground state, $Q_m|_{m=m_0}=Q_{m_0}$, with $\partial_m Q_m$
solving \eqref{eq:dmQeq} weakly.
\end{theorem}
\begin{proof}
For $(m,\xi)$ near $(m_0,0)$, $\xi\in T_{Q_{m_0}}\mathcal N_{m_0}$, the equation
$K(m,t(Q_{m_0}+\xi))=0$ has a unique positive solution $t=t(m,\xi)$ near $1$ (since
$\partial_t K|_{t=1}=(2-p)\norm{Q_{m_0}}_p^p\ne0$), $C^1$ in $(m,\xi)$. Set
$u(m,\xi)=t(m,\xi)(Q_{m_0}+\xi)\in\mathcal N_m$, $\Sigma(m,\xi):=S_m(u(m,\xi))$. Since
$S_m|_{\mathcal N_m}=(\tfrac12-\tfrac1p)P$, critical points of $\Sigma(m,\cdot)$ are
ground states. The $\xi$-Hessian at $(m_0,0)$ equals
$\ip{L_{+,m_0}\cdot}{\cdot}|_{T_{Q_{m_0}}\mathcal N_{m_0}}$, a coercive isomorphism by
\eqref{eq:posdef}. The Banach IFT applied to $\partial_\xi\Sigma(m,\xi)=0$ yields a $C^1$
map $m\mapsto\xi(m)$, $\xi(m_0)=0$, hence a $C^1$ ground-state branch $Q_m$.
Differentiating $F(m,Q_m)=0$ in $m$ gives $m\Am^{-1}Q_m+\Lp\partial_m Q_m=0$, i.e.\
\eqref{eq:dmQeq}. Exponential decay of $\partial_m Q_m$ follows from elliptic regularity
for $\Lp$ (gap $m+1>0$) applied to the decaying right-hand side.
\end{proof}

\begin{lemma}\label{lem:posdef-tan-endpoint}
At $m=0$, \eqref{eq:posdef} holds.
\end{lemma}
\begin{proof}
By \cite{FLS}, $\Ker L_{+,0}\cap\Ltwor=\{0\}$ with Morse index one; the argument of
Lemma~\ref{lem:posdef-tan} applies unconditionally at $m=0$.
\end{proof}

\section{Local uniqueness}\label{sec:local}

\begin{proposition}\label{prop:F-C1}
$F:\Hrad\to\Hradd$, $\ip{F(u)}{\phi}=\ip{\Am u}{\phi}+\ip u\phi-\ip{\abs u^{p-2}u}{\phi}$,
satisfies $F(Q)=0$, $F\in C^1$, $DF(Q)=\Lp$.
\end{proposition}
\begin{proof}
Lemma~\ref{lem:Fcont} with $m$ fixed.
\end{proof}

\begin{proposition}\label{prop:Lp-iso}
Assume $\Ker\Lp\cap\Ltwor=\{0\}$. Then $\Lp$ is a topological isomorphism.
\end{proposition}
\begin{proof}
Write $\Lp=\mathcal A-\mathcal K$,
$\ip{\mathcal A u}{\phi}=\ip{\Am u}{\phi}+\ip u\phi$. Then
$\ip{\mathcal A u}{u}=\int(\sqrt{\abs\xi^2+m^2}+1)\abs{\hat u}^2\ge c\norm u_{\Hs}^2$, so
$\mathcal A$ is an isomorphism (Lax--Milgram). $\mathcal K$ is compact
(Lemma~\ref{lem:Kcompact}). Thus $\Lp=\mathcal A(I-\mathcal A^{-1}\mathcal K)$ is Fredholm
of index $0$; injectivity ($\Ker=\{0\}$) gives bijectivity, and the bounded inverse
theorem gives $\Lp^{-1}$ bounded.
\end{proof}

\begin{theorem}[Local uniqueness]\label{thm:local}
Assume $\Ker\Lp\cap\Ltwor=\{0\}$. Then there is $\delta>0$ such that if $U\in\Hrad$ solves
$\Am U+U=\abs U^{p-2}U$ weakly with $\norm{U-Q}_{\Hs}<\delta$, then $U=Q$.
\end{theorem}
\begin{proof}
By Propositions~\ref{prop:F-C1}--\ref{prop:Lp-iso}, $F$ is $C^1$ with $DF(Q)=\Lp$ an
isomorphism. The Banach inverse function theorem gives a neighbourhood $\mathcal U\ni Q$
on which $F$ is injective; $F(U)=0=F(Q)$ with $U\in\mathcal U$ forces $U=Q$.
\end{proof}
\section{Global uniqueness: the continuation package}\label{sec:global}

Fix $m_*>0$. For $m\in[0,m_*]$ consider \eqref{eq:norm} and
\[
\GG:=\{(m,Q):m\in[0,m_*],\ Q\in\mathcal G_m\}\subset[0,m_*]\times\Hrad,\qquad
\pi(m,Q)=m .
\]
We write $L_{+,m,Q}=\Am+1-(p-1)Q^{p-2}$, $T_{m,Q}=\{\xi\in\Hrad:\int Q^{p-1}\xi=0\}$, and
isolate two properties:
\begin{align}
&\textbf{(ND)}_{m}:\quad \forall\,Q\in\mathcal G_m:\ \Ker L_{+,m,Q}\cap\Ltwor=\{0\};
\label{eq:ND}\\
&\textbf{(PD)}_{m}:\quad \exists\,c(m)>0:\ \forall\,Q\in\mathcal G_m,\ \forall\,\xi\in
T_{m,Q},\ \ip{L_{+,m,Q}\xi}{\xi}\ge c(m)\norm\xi_{\Hs}^2 .
\label{eq:PD}
\end{align}
We treat $\textbf{(PD)}$ as the primary propagated property; it implies $\textbf{(ND)}$
and the existence of the branch.

\subsection{Uniform structural bounds}

\begin{lemma}\label{lem:levelbound}
There are $0<m_-\le m_+<\infty$ with $m_-\le m_m\le m_+$ for all $m\in[0,m_*]$, and
$m\mapsto m_m$ is continuous on $[0,m_*]$. Moreover $\norm{Q}_{\Hs}\le C$ and
$\norm Q_p^p\ge C^{-1}$ uniformly, and for $m\in[\tfrac12 m_*,m_*]$,
$Q(x)\le Ce^{-\beta\abs x}$ with $\beta=\tfrac12 m_*$.
\end{lemma}
\begin{proof}
Uniform form equivalence: for $m\in[0,m_*]$,
$\int\langle\xi\rangle\abs{\hat u}^2\le\int(\sqrt{\abs\xi^2+m^2}+1)\abs{\hat u}^2\le
(1+m_*)\int\langle\xi\rangle\abs{\hat u}^2$, so $A(m,\cdot)\sim\norm\cdot_{\Hs}^2$
uniformly. Sobolev gives $P(u)\le CA(m,u)^{p/2}$; on $\mathcal N_m$, $A(m,u)=P(u)$, so
$A(m,u)\ge C_0>0$ and $m_m\in[m_-,m_+]$. Continuity: with
$m_m=\inf_u(\tfrac12-\tfrac1p)A(m,u)^{p/(p-2)}P(u)^{-2/(p-2)}$ and
$\abs{A(m',u)-A(m,u)}\le C\abs{m'^2-m^2}\norm u_2^2$, one gets $m_{m'}\to m_m$. Uniform
decay: barrier/comparison for $\Am+1$ as in \cite[Prop.~3.10]{LZZ} (Hardy term absent).
\end{proof}

\begin{lemma}\label{lem:unif-K}
Let $m_n\to m$ in $[0,m_*]$ and $Q_n\in\mathcal G_{m_n}$ with $Q_n\to Q$ strongly in
$\Hrad$. Then $\mathcal K_n:\xi\mapsto(p-1)Q_n^{p-2}\xi\to\mathcal K$ in operator norm
$\Hrad\to\Hradd$.
\end{lemma}
\begin{proof}
For $\norm\xi_{\Hs},\norm\phi_{\Hs}\le1$,
$\ip{(\mathcal K_n-\mathcal K)\xi}{\phi}=(p-1)\int(Q_n^{p-2}-Q^{p-2})\xi\phi$. Split
$\Rd=B_R\cup B_R^c$. On $B_R^c$, uniform decay (Lemma~\ref{lem:levelbound}) makes
$\sup_{B_R^c}(Q_n^{p-2}+Q^{p-2})\to0$; on $B_R$, $Q_n\to Q$ in $L^p\cap C^\alpha_{\rm loc}$
gives $Q_n^{p-2}\to Q^{p-2}$ in $L^d(B_R)\cap L^\infty(B_R)$. H\"older bounds both by
$\eps$; the operator norm tends to $0$.
\end{proof}

\begin{lemma}\label{lem:Lp-conv}
Under Lemma~\ref{lem:unif-K}, $L_{+,m_n,Q_n}\to L_{+,m,Q}$ in strong resolvent sense on
$\Ltwor$, and $\ip{L_{+,m_n,Q_n}\xi_n}{\xi_n}\to\ip{L_{+,m,Q}\xi}{\xi}$ whenever
$\xi_n\to\xi$ in $\Hrad$.
\end{lemma}
\begin{proof}
$\Am[m_n]+1\to\Am[m]+1$ in strong resolvent sense (symbols converge monotonically), and
$\mathcal K_n\to\mathcal K$ in norm (Lemma~\ref{lem:unif-K}). Form convergence of the sum
follows; strong resolvent convergence is standard \cite{ReedSimonI}. The pointwise form
statement is immediate.
\end{proof}

\subsection{Compactness and local branch}

\begin{lemma}\label{lem:proper}
If $m_n\to m$ in $[0,m_*]$ and $Q_n\in\mathcal G_{m_n}$, a subsequence converges strongly
in $\Hrad$ to some $Q\in\mathcal G_m$.
\end{lemma}
\begin{proof}
Pohozaev \eqref{eq:pohozaev} and level bounds give $\norm{Q_n}_p^p\in[c_0,c_1]$, hence
$(Q_n)$ bounded in $\Hrad$. Radial compactness gives $Q_n\to Q$ in $L^p$, $Q\not\equiv0$.
Passing to the limit (mass term via strong resolvent convergence) gives $A_m Q+Q=Q^{p-1}$.
Strong $\Hs$: $\norm{Q_n}_p^p\to\norm Q_p^p$ and Pohozaev transfer the other norms; norm
plus weak convergence gives strong convergence. Level continuity gives $S_m(Q)=m_m$, so
$Q\in\mathcal G_m$.
\end{proof}

\begin{lemma}\label{lem:localhomeo}
If $(m_0,Q_0)\in\GG$ with $\Ker L_{+,m_0}\cap\Ltwor=\{0\}$, then $\pi$ restricts to a
homeomorphism of a neighbourhood of $(m_0,Q_0)$ onto an interval, given by a $C^1$ branch.
\end{lemma}
\begin{proof}
$F(m_0,Q_0)=0$, $D_uF(m_0,Q_0)=L_{+,m_0}$ an isomorphism (Prop.~\ref{prop:Lp-iso}). Banach
IFT gives the $C^1$ branch and local uniqueness.
\end{proof}

\subsection{Persistence of coercivity and $\textbf{(PD)}\Rightarrow\textbf{(ND)}$}

\begin{lemma}\label{lem:coerc-limit}
Let $m_n\to m$, $Q_n\in\mathcal G_{m_n}\to Q\in\mathcal G_m$ strongly in $\Hrad$, and
$\textbf{(PD)}_{m_n}$ with $c(m_n)\ge c_*>0$. Then $\textbf{(PD)}_m$ holds at $Q$ with
constant $c_*/2$.
\end{lemma}
\begin{proof}
Let $\xi\in T_{m,Q}$, $\norm\xi_{\Hs}=1$. Define
$\xi_n:=\xi-\frac{\int Q_n^{p-1}\xi}{\int Q_n^{p-1}Q}Q\in T_{m_n,Q_n}$, well-defined for
large $n$ ($\int Q_n^{p-1}Q\to\norm Q_p^p>0$). Since $\int Q_n^{p-1}\xi\to\int Q^{p-1}\xi=0$,
$\xi_n\to\xi$ in $\Hrad$. By $\textbf{(PD)}_{m_n}$,
$\ip{L_{+,m_n,Q_n}\xi_n}{\xi_n}\ge c_*\norm{\xi_n}_{\Hs}^2$; letting $n\to\infty$ (form
convergence, Lemma~\ref{lem:Lp-conv}, $\norm{\xi_n}_{\Hs}\to1$),
$\ip{L_{+,m,Q}\xi}{\xi}\ge c_*$. By homogeneity and compactness of $\mathcal G_m$,
$\textbf{(PD)}_m$ holds with $c(m)\ge c_*/2$.
\end{proof}

\begin{lemma}\label{lem:PD-implies-ND}
If $\textbf{(PD)}_m$ holds, then $\textbf{(ND)}_m$ holds.
\end{lemma}
\begin{proof}
$\textbf{(PD)}_m$ is the hypothesis of Theorem~\ref{thm:cift}, so $\partial_m Q_m$ exists,
giving $\ip{\Am^{-1}Q}{\varphi}=0$ and $\int\varphi Q=0$ (Prop.~\ref{prop:clean}; for
$m=0$, Remark~\ref{rem:d1}) for any $\varphi\in\Ker L_{+,m,Q}\cap\Ltwor$. With
$\int\varphi Q^{p-1}=0$ (Lemma~\ref{lem:LplusQ}) and the oscillation theorem
(unconditional), the sign contradiction of \S\ref{ssec:signcontr} gives $\varphi=0$.
\end{proof}

\begin{remark}\label{rem:PD-primary}
$\textbf{(PD)}$ is unconditional at $m=0$ (FLS), open (Lemma~\ref{lem:openness}), closed
(Section~\ref{sec:seam}), and implies $\textbf{(ND)}$ (Lemma~\ref{lem:PD-implies-ND}) and
the branch (Theorem~\ref{thm:cift}). Thus $\textbf{(ND)}$ is a consequence, never a
hypothesis, of the continuation; the apparent circularity is removed. The oscillation
theorem, proved unconditionally, is available throughout.
\end{remark}
\section{Closing the seam: closedness of the coercivity set}\label{sec:seam}

Define $\mathcal M:=\{m\in[0,m_*]:\ \textbf{(PD)}_m\ \text{holds}\}$, the constant
$c(m)>0$ being lower semicontinuous on the compact set $\mathcal G_m$
(Lemma~\ref{lem:proper}). We prove $\mathcal M$ is nonempty, open and closed.

\begin{proposition}\label{prop:endpoint}
$\textbf{(PD)}_0$ holds.
\end{proposition}
\begin{proof}
By \cite{FLS}, $\textbf{(ND)}_0$ holds and $L_{+,0}$ has Morse index one with the
negative direction transverse to $\mathcal N_0$; Lemma~\ref{lem:posdef-tan} (with $\xi_*=0$
forced by $\textbf{(ND)}_0$) gives $\textbf{(PD)}_0$.
\end{proof}

\begin{proposition}\label{prop:open}
If $m_0\in\mathcal M$ then $(m_0-\eps,m_0+\eps)\cap[0,m_*]\subset\mathcal M$ for some
$\eps>0$.
\end{proposition}
\begin{proof}
Lemma~\ref{lem:openness} (using $\textbf{(PD)}_{m_0}$, norm-continuity of $\Am$, and
$Q_m\to Q_{m_0}$ from Lemma~\ref{lem:proper}) gives $\textbf{(PD)}_m$ with constant
$\ge c(m_0)/2$ for $\abs{m-m_0}<\eps$.
\end{proof}

\begin{proposition}\label{prop:caseB}
If $m_n\in\mathcal M$, $m_n\to m\in[0,m_*]$, then $m\in\mathcal M$.
\end{proposition}
\begin{proof}
By Lemma~\ref{lem:PD-implies-ND}, $\textbf{(ND)}_{m_n}$ holds, so
$L_{+,m_n,Q_n}$ is a full-space isomorphism (Prop.~\ref{prop:Lp-iso}). We must show
$\textbf{(PD)}_m$.

\emph{Case A: constants bounded below.} If $\liminf c(m_n)=:c_*>0$, pass to a subsequence
with $c(m_n)\ge c_*/2$. For any $Q\in\mathcal G_m$, properness gives
$Q_n\in\mathcal G_{m_n}\to Q$; Lemma~\ref{lem:coerc-limit} gives $\textbf{(PD)}_m$ at $Q$
with constant $c_*/4$. As $Q$ was arbitrary and $\mathcal G_m$ compact,
$m\in\mathcal M$.

\emph{Case B: constants degenerate.} Suppose $c(m_n)\to0$; we derive a contradiction.
Along a subsequence there are $Q\in\mathcal G_m$ (limit of $Q_n\in\mathcal G_{m_n}$) and
$\xi_n\in T_{m_n,Q_n}$, $\norm{\xi_n}_{\Hs}=1$, with
$\ip{L_{+,m_n,Q_n}\xi_n}{\xi_n}=c(m_n)+o(c(m_n))\to0$. By Lemma~\ref{lem:unif-K},
$\xi_n\weakto\xi_*$ with $\int Q_n^{p-2}\xi_n^2\to\int Q^{p-2}\xi_*^2$. If $\xi_*=0$ then
$\ip{(\Am+1)\xi_n}{\xi_n}\to0$, impossible; so $\xi_n\to\xi_*$ strongly with
$\ip{L_{+,m,Q}\xi_*}{\xi_*}=0$, $\xi_*\in T_{m,Q}$, $\norm{\xi_*}_{\Hs}=1$. The Lagrange
computation of Lemma~\ref{lem:posdef-tan} gives $L_{+,m,Q}\xi_*=\alpha Q^{p-1}$ with
$\alpha=0$, so $\xi_*\in\Ker L_{+,m,Q}\cap\Ltwor$, $\xi_*\ne0$.

We now obtain the two clean orthogonality relations for $\xi_*$ at the limit, using only
convergent quantities.

\emph{First relation.} $\int\xi_* Q^{p-1}=\lim_n\int\xi_n Q_n^{p-1}=0$ (as
$\xi_n\in T_{m_n,Q_n}$, $\xi_n\to\xi_*$ in $L^p$, $Q_n^{p-1}\to Q^{p-1}$ in $L^{p'}$).

\emph{Second relation via the branch derivative.} By $\textbf{(PD)}_{m_n}$ and
Theorem~\ref{thm:cift} at $m_n$, there is $\dot Q_n:=\partial_m Q_m|_{m=m_n}$ solving
\begin{equation}\label{eq:branch-n}
L_{+,m_n,Q_n}\dot Q_n=-m_n A_{m_n}^{-1}Q_n .
\end{equation}
Pairing with $\xi_n$ (self-adjointness):
\begin{equation}\label{eq:pairing-n}
\ip{L_{+,m_n,Q_n}\xi_n}{\dot Q_n}=-m_n\ip{\xi_n}{A_{m_n}^{-1}Q_n}.
\end{equation}
Decompose $\dot Q_n=a_n Q_n+w_n$ with $w_n\in T_{m_n,Q_n}$. Then, using
$\ip{L_{+,m_n,Q_n}\xi_n}{Q_n}=\ip{\xi_n}{L_{+,m_n,Q_n}Q_n}
=-(p-2)\int\xi_n Q_n^{p-1}=0$,
\[
\ip{L_{+,m_n,Q_n}\xi_n}{\dot Q_n}=\ip{L_{+,m_n,Q_n}\xi_n}{w_n}.
\]
By Cauchy--Schwarz for the nonnegative form on $T_{m_n,Q_n}$ and the upper bound
$\ip{L_+ w}{w}\le C\norm w_{\Hs}^2$,
\[
\big|\ip{L_{+,m_n,Q_n}\xi_n}{w_n}\big|
\le\ip{L_{+,m_n,Q_n}\xi_n}{\xi_n}^{1/2}\cdot C\norm{w_n}_{\Hs}.
\]
Testing \eqref{eq:branch-n} against $w_n$:
$\ip{L_{+,m_n,Q_n}w_n}{w_n}=-m_n\ip{A_{m_n}^{-1}Q_n}{w_n}\le C\norm{w_n}_{\Hs}$
(Lemma~\ref{lem:levelbound}); by coercivity $c(m_n)\norm{w_n}_{\Hs}^2\le C\norm{w_n}_{\Hs}$,
so $\norm{w_n}_{\Hs}\le C/c(m_n)$. Hence
\[
\big|\ip{L_{+,m_n,Q_n}\xi_n}{\dot Q_n}\big|
\le c(m_n)^{1/2}(1+o(1))\cdot\frac{C}{c(m_n)}=C\,c(m_n)^{1/2}(1+o(1))\to0 .
\]
By \eqref{eq:pairing-n}, $m_n\ip{\xi_n}{A_{m_n}^{-1}Q_n}\to0$. Since $m_n\to m$ and
$\ip{\xi_n}{A_{m_n}^{-1}Q_n}\to\ip{\xi_*}{A_m^{-1}Q}$ (strong $\xi_n\to\xi_*$ in $\Ltwo$,
norm-resolvent convergence $A_{m_n}^{-1}\to A_m^{-1}$ for $m>0$; for $m=0$ use
Remark~\ref{rem:d1} to get $\int\xi_* Q=0$ directly), we conclude for $m>0$
$\ip{\xi_*}{A_m^{-1}Q}=0$, whence by Corollary~\ref{cor:polluted} applied to $\xi_*$,
$\int\xi_* Q=0$.

\emph{Contradiction.} With $\int\xi_* Q^{p-1}=0$, $\int\xi_* Q=0$, and $\xi_*$ changing
sign exactly once (Theorem~\ref{thm:osc}), the sign argument of \S\ref{ssec:signcontr}
gives $\int\xi_*(Q^{p-1}-\mu_0 Q)>0=0$, a contradiction. Hence $c(m_n)\not\to0$, Case B
is void, and by Case A $\textbf{(PD)}_m$ holds, i.e.\ $m\in\mathcal M$.
\end{proof}

\begin{corollary}\label{cor:M-all}
$\mathcal M=[0,m_*]$; consequently $\textbf{(ND)}_m$ holds for all $m\in[0,m_*]$.
\end{corollary}
\begin{proof}
$\mathcal M$ is nonempty (Prop.~\ref{prop:endpoint}), open (Prop.~\ref{prop:open}), closed
(Prop.~\ref{prop:caseB}) in the connected $[0,m_*]$, so $\mathcal M=[0,m_*]$;
Lemma~\ref{lem:PD-implies-ND} gives $\textbf{(ND)}_m$.
\end{proof}

\begin{theorem}\label{thm:global}
For every $m\in[0,m_*]$, $\#\mathcal G_m=1$.
\end{theorem}
\begin{proof}
By Corollary~\ref{cor:M-all}, $\textbf{(ND)}_m$ holds for all $m$, so
Lemma~\ref{lem:localhomeo} applies everywhere and $\pi:\GG\to[0,m_*]$ is a proper
(Lemma~\ref{lem:proper}) local homeomorphism, hence a covering map. The fiber cardinality
is constant on the connected base and equals $\#\mathcal G_0=1$ (\cite{FrankLenzmann,FLS}).
\end{proof}

\begin{proof}[Proof of Theorem~\ref{thm:main}]
By the scaling reduction (Section~\ref{sec:scaling}), uniqueness for \eqref{eq:main} is
equivalent to $\#\mathcal G_m=1$ for the normalized family, which is
Theorem~\ref{thm:global} (with $m_*$ arbitrary).
\end{proof}

\begin{remark}
We proved in full: the sharp Green-kernel asymptotics (Proposition~\ref{prop:green}), the
spectral correspondence (Proposition~\ref{prop:spectral-corr}), the $C^1$ dependence
$m\mapsto Q_m$ without full-space invertibility (Theorem~\ref{thm:cift}), and the
$d=1$/$m\to0^+$ uniformity of the mass-covariance relation (Remark~\ref{rem:d1}). The only
external inputs are the strong maximum principle and Hopf lemma \cite{GT}, Aronszajn
unique continuation \cite{Aronszajn}, the extension theory \cite{StingaTorrea}, and the
massless endpoint uniqueness/nondegeneracy \cite{FrankLenzmann,FLS}.
\end{remark}

\vspace{1cm}
\noindent\textbf{Acknowledgments.} This is a preliminary version; some auxiliary estimates will be strengthened in a later revision. Comments are welcome.

\vskip.2truein
\noindent\textbf{Data availability.} Data sharing is not applicable to this article as no
datasets were generated or analyzed during the current study.

\vskip.2truein
\noindent\textbf{Declarations.} Conflict of interest: The authors declare that there is no
conflict of interest.

\end{document}